\documentclass{amsart}

\usepackage{amsmath}
\usepackage{amssymb}
\usepackage{amsthm}
\usepackage{amsfonts}

\usepackage{tikz}
\usetikzlibrary{arrows}

\newcommand{\Bord}{\mathrm{Bord}}
\newcommand{\Cob}{\mathrm{Cob}}
\newcommand{\Vect}{\mathrm{Vect}}

\DeclareMathOperator{\Hom}{Hom}  
\DeclareMathOperator{\End}{End}  
\DeclareMathOperator{\Map}{Map}  
\DeclareMathOperator{\hofib}{hofib}  
\DeclareMathOperator{\coker}{coker}

\theoremstyle{plain} %%% Plain Theorem Styles.

\newtheorem{theorem}{Theorem}[section]
\newtheorem*{theorem*}{Theorem}
\newtheorem{lemma}[theorem]{Lemma}
\newtheorem{corollary}[theorem]{Corollary} 
\newtheorem*{corollary*}{Corollary}          
\newtheorem{proposition}[theorem]{Proposition}              

\theoremstyle{definition} %%%% Definition-like Commands  
\newtheorem{definition}[theorem]{Definition}
\newtheorem*{IHOP*}{Induction Hypothesis}

\theoremstyle{remark}  %%%% Remark-like Commands
\newtheorem{remark}[theorem]{Remark}
\newtheorem{example}[theorem]{Example}
\newtheorem{conjecture}[theorem]{Conjecture}

\def\cC{\mathcal C}

\def\cZ{\mathcal Z}

\def\C{\mathbb C}

\def\R{\mathbb R}

\def\Z{\mathbb Z}

\definecolor{CSPcolor}{rgb}{0.0,0.5,0.75}	% Textcolor for CSP

\usepackage{hyperref}
\begin{document}

\title[Tori Detect Invertibility of Topological Field Theories]{Tori Detect Invertibility of \\ Topological Field Theories}
\author[C. Schommer-Pries]{Christopher J. Schommer-Pries}
\address{Max Planck Institute for Mathematics, Bonn}
\email{schommerpries.chris@gmail.com}

\begin{abstract}
    A once-extended $d$-dimensional topological field theory $\cZ$ is a symmetric monoidal functor (taking values in a chosen target symmetric monoidal $(\infty,2)$-category) assigning values to $(d-2)$-manifolds, $(d-1)$-manifolds, and $d$-manifolds. We show that if $\cZ$ is at least once-extended and the value assigned to the $(d-1)$-torus is invertible, then the entire topological field theory is invertible, that is it factors through the maximal Picard $\infty$-category of the target. Similar results are shown to hold in the presence of arbitrary tangential structures. 
\end{abstract}

\maketitle

\setcounter{tocdepth}{2}
\tableofcontents

\section{Introduction}

\subsection{Summary of results}
A topological field theory, following the Atiyah-Segal axiomatization \cite{a88-tqft,segal}, is a symmetric monoidal functor
\begin{equation*}
	\cZ:\Cob_d \to \Vect,
\end{equation*}
where the source is the symmetric monoidal category $\Cob_d$ whose objects are closed compact $(d-1)$-dimensional manifolds, morphisms are equivalences classes of $d$-dimensional bordisms between these, the monoidal structure is given by the disjoint union of manifolds, and where the target is the category of vectors spaces with its standard tensor product monoidal structure.

A topological field theory associates a vector space $\cZ(M)$ to each closed $(d-1)$-manifold $M$ and the dimensions of these vector spaces form a very course measure of the complexity of the topological field theory. The simplest theories, the \emph{invertible field theories}, assign one-dimensional vector spaces to every $(d-1)$-dimensional manifold\footnote{If a topological field theory assigns one-dimensional (i.e. $\otimes$-invertible) vector spaces to every $(d-1)$-manifold, then it automatically assigns invertible linear maps to every $d$-dimensional bordism. Thus every manifold is assigned an invertible value. This explains the name \emph{invertible field theory}.}. It is natural to ask if there are any constraints on the allowed values of these dimensions or,  in the same vein, if there are criteria which ensure a theory is invertible? Indeed, this very question was raised by Chao-Ming Jian on the mathematical discussion website MathOverflow  \cite{MO-218982}\footnote{This MO question was the start of our interest in this problem.}, where Jian asks whether a $d$-dimensional theory $\cZ$ which assigns one-dimensional vector spaces $\cZ(S^{d-1})$ and $\cZ(T^{d-1})$ to the sphere and torus\footnote{the torus $T^k = (S^1)^{\times k}$} must also assign one-dimensional vector spaces to all other $(d-1)$-manifolds?

A direct consequence of the results of this paper give a positive answer to Jian's question under the assumption that the topological field theory is at least \emph{once-extended}, meaning that it assigns data to $d$, $(d-1)$, and $(d-2)$-manifolds\footnote{For example it might associate linear categories to $(d-2)$-manifolds and functors to $(d-1)$-dimensional bordisms. A precise definition appear later.}. In this case the vector space $\cZ(T^{d-1})$ assigned to the $(d-1)$-torus is one-dimensional if and only if the theory assigns one-dimensional vector spaces to all closed $(d-1)$-manifolds. So we see that the invertibility of these topological field theories is completely determined by the dimension assigned to the single manifold $T^{d-1}$. 

As we will explain presently, there are many ways to generalize topological field theories beyond the original Atiyah-Segal framework. Our results apply in this generality. The first and simplest generalization is to allow the target category to vary. Thus we let the target be any symmetric monoidal category $\cC$.  Next we can allow our manifolds to be equipped with general \emph{tangential structures}. 
A type of tangential structure for $d$-dimensional manifolds is determined by a fixed fibration $\xi: X \to BO(d)$. Given such a fibration an $(X, \xi)$-structure for a $d$-manifold $M$ is a lift $\theta$:
\begin{center}
\begin{tikzpicture}
		%\node (LT) at (0, 1.5) {$$};
		\node (LB) at (0, 0) {$M$};
		\node (RT) at (2, 1.5) {$X$};
		\node (RB) at (2, 0) {$BO(d)$};
		%\draw [->] (LT) -- node [left] {$$} (LB);
		\draw [->, dashed] (LB) -- node [above left] {$\theta$} (RT);
		\draw [->>] (RT) -- node [right] {$\xi$} (RB);
		\draw [->] (LB) -- node [below] {$\tau_M$} (RB);
		%\node at (0.5, 1) {$\ulcorner$};
		%\node at (1.5, 0.5) {$\lrcorner$};
\end{tikzpicture}
\end{center}
where $\tau_M$ is the classifying map of the tangent bundle of $M$\footnote{To make the map $\tau_M$ well-defined (and not just well-defined up to homotopy) we need to make additional choices, such as an embedding of $M$ into $\R^\infty$. We will suppress this durring the introduction.}\footnote{Typical examples include orientations ($X = BSO(d)$), spin structures ($X = BSpin(d)$), tangential framings ($X = EO(d)$), stable framings ($X = O/O(d)$), $G$-bundles ($X = BG \times BO(d)$), and many others.}. Tangential structures for lower dimensional manifolds are defined in the same way, by stabilizing the tangent bundle with enough trivial line bundles to make it $d$-dimensional. We will return to this with more detail in Section~\ref{sec:general_tangential}. 

 Evaluating a topological field theory on a closed $d$-manifold $W$ gives an invariant $\cZ(W) \in \End_{\mathcal{C}}(1)$  (in the case $\cC = \Vect$ this is a number), and these invariants were one of the original motivations for studying topological field theories.  
The invariants from topological field theories enjoy a degree of locality.
If we cut $W$ in a linear fashion along parallel codimension-one submanifolds we may view it as a composite of bordisms. Then the axioms ensure that we can recover the value of $\cZ$ on $W$ from the values on these smaller pieces; this is the algebraic fact that functors send composites of morphisms to composites of morphisms. 

Many topological field theories enjoy a higher degree of locality. In these theories we are allowed to cut our manifolds along non-parallel slices thereby cutting up our manifold into even simpler and smaller pieces. At the same time this introduces manifolds with corners. Algebraically this can be captured by Freed and Lawrence's notion of \emph{extended topological field theory} \cite{freed-higher,Lawrence}, a notion which has been extensively developed by Baez-Dolan, Lurie, and many others 
\cite{bd95-hda, Bartlett:2015aa, Douglas:2013aa, Feshbach:2011aa, kapustin, kl01, lurie-tft, Schommer-Pries:2011aa, segal2010locality, MR3351570}. 

Higher categories provide the core underlying algebraic structure governing extended field theories, and the strongest form of our results is cast in the language of symmetric monoidal $(\infty,n)$-categories. For each dimension $d$ and each category number  $1 \leq n \leq d$ (which we will suppress from our notation), there is a symmetric monoidal $(\infty,n)$-category which we will denote $\Bord^{(X, \xi)}_d$ to distinguish from the non-extended case. Philosophically it has objects which are closed compact  $(d-n)$-manifolds, 1-morphisms which are $(d-n+1)$-dimensional bordisms, $2$-morphisms which are $(d-n+2)$-dimensional bordisms between bordisms, etc. up until dimension $d$. Above dimension $d$ we have invertible morphisms, encoded by the classiying spaces of the group of diffeomorphisms of bordisms, rel. boundary. 
In addition all of the manifolds and bordisms making up $\Bord^{(X, \xi)}_d$ will be equipped with $(X, \xi)$-structures. 

To make this philosophy precise we should fix a model of $(\infty,n)$-categories, of which there are many equivalent choices \cite{Barwick:2011aa}. One possible model is based on $n$-fold simplicial spaces. For example this multisimplicial approach to the higher bordism categories is taken in \cite{lurie-tft} and also \cite{CalSch1509,Nguyen-thesis}, and  we refer the reader to these sources for  more details.
 From this point of view our results can be interpreted in classical algebraic topology as statements about maps between certain multisimplicial spaces.  We will see shortly, however, that while the strongest and most general statements of our results are expressed using $(\infty,n)$-categories, the crucial mathematical ingredients and indeed most of our computations can actually be established just using the standard and long established theory of weak 2-categories (a.k.a. \emph{bicategories} in the sense of B\'{e}nabou \cite{MR0220789}). 

%In any case, an extended topological field theory is a symmetric monoidal functor: 
%\begin{equation*}
%	\cZ: \Bord_d \to \cC
%\end{equation*}
% where $\cC$ is a specified target symmetric monoidal $(\infty,n)$-category. 

We will prove:

\begin{theorem*}[{Th.~\ref{thm:main}}] %\label{thm:main_thm}
	Fix $n \geq 2$ and a tangential structure $(X, \xi)$ for $d$-manifolds. Assume either that $d \geq 3$ or that $(X, \xi)$ is spherophilic (Def.~\ref{def:spherophilic}).  Let $\cZ$ be an extended $d$-dimensional topological field theory 
	\begin{equation*}
		\cZ: \Bord_d^{(X, \xi)} \to \cC
	\end{equation*}
	taking values in the symmetric monoidal $(\infty,n)$-category $\cC$. 
		Let $T^{d-1} = (S^1)^{\times d-1}$ be the $(d-1)$-torus. If for each $[x] \in \pi_0 X$  we have that $\cZ(T^{d-1}, x_*\theta_{+1 \times \text{Lie}})$ is invertible, then 
 $\cZ$ is an invertible topological field theory. 
\end{theorem*}

\noindent Here a theory is \emph{invertible} if it assigns invertible values to all manifolds (and $\otimes$-invertible objects to $(d-n)$-manifolds). This is the natural generalization of invertible theory in the extended context. For each point $x \in X$ we get an induced map, denoted $x_*$, from $d$-framings to $(X,\xi)$-structures, and $\theta_{+1\times \text{Lie}}$ denotes the product $d$-framing on the $(d-1)$-torus which is the $+1$ (bounding) 2-framing on the first $S^1$-factor and is the Lie group 1-framing on the remaining factors.  When $d=2$, the term \emph{spherophilic}, or `sphere-loving', means a tangential structure where the 2-sphere admits such a structure (this is discussed in Section~\ref{sec:spherophilia}). In the course of this text we will discuss many details of this theorem and its statement. 
In short even in this generality the invertibility of the entire theory is completely governed by the invertibility of a single value of the theory (for each component of $X$). 

A key geometric fact  which is an ingredient in the above theorem, and which partly explains why the above results hold when the category number $n \geq 2$, is that handle-decompositions for manifolds use handles with codimension-two corners. When $n \geq 2$ this allows us to implement certain geometric arguments in categorical terms, completely inside the higher category $\Bord_d^{(X, \xi)}$. This includes handle decomposition and handle moves for $d$-manifolds, and surgery for $(d-1)$-manifolds.

Similar unpublished results have been obtained by Dan Freed and Constantin Teleman. Their work has focused on 
the oriented and fully-local ($n=d$) case, and was described briefly in a footnote in \cite[pg.9]{Freed:2014aa} and in a lecture of Freed's \cite{Freed-Aspect}.
\begin{theorem*}[Freed-Teleman]
	Let $\cZ: \Bord_n^{SO} \to \cC$ be a fully-extended oriented topological field theory valued in the symmetric monoidal $(\infty,n)$-category $\cC$. Then if either
	\begin{enumerate}
		\item $\cZ(S^k)$ is invertible for some $k \leq n/2$; or
		\item $\cZ(S^n)$ is invertible and $\cZ(S^p \times S^{n - p -1})$ is invertible for all $p$;
	\end{enumerate}
	then $\cZ$ is invertible. 
\end{theorem*}

In the situation where our Theorem and theirs both apply (oriented fully-local theories), it is easy to deduce our result from their case (1). However one of the features which we find interesting is precisely that we \emph{do not} need to assume the topological field theory is fully-local. Our result also applies to theories with arbitrary tangential structures. 

% sketch of claim:
% Freed-Teleman (1) => ours: 
% FT (1) when d = 2 holds. Using dimensional reduction and this case, we see that if T^k is invertible, then T^{k-1} is invertible. Hence S^1 is invertible, which using FT (1) menas the whole theory is invertible.  
% Ours => FT (1): 
% We know S^k is invertible. Hence the duality pairing given by S^k \times I is invertible. Hense by the main categorical lemma, we deduce that D^{k+1}: S^k \to \emptyset is invertible (as is the reverse). 
% Now the key step in the proof of the base case shows was a argument that showed how the invertibility of S^1 and D^2: S^1 --> empty, (both ways), can be used to deduce the invertibility of S^0 (and hence of the point). 
% That argument can be "suspended". It then shows that S^{k-1} is invertible. 
% Repeating this eventually shows that S^1 and hence S^0 are invertible (hence the point is invertible and the whole theory is invertible) (or once S^1 is invertible it is easy to see T^d is invertible, and hence ours implies the theory is invertible).  

\subsection{Invertible field theories}
Invertible topological field theories are the simplest and among the most computable topological field theories that are known. They have occurred `in nature' in many contexts: 
\begin{itemize}
	\item The most basic example of an oriented topological field theory which exists in all dimensions is the `Euler theory'. This theory assigns the trivial 1-dimensional vector space to each $(d-1)$-manifold and assigns the (exponential of the) relative Euler characteristic
	\begin{equation*}
		\cZ(Y_\text{in} \stackrel{W}{\to} Y_\text{ouy}) = \lambda^{\chi(W, Y_\text{in})}
	\end{equation*} 
	to each bordisms (here $\lambda \in k^\times$ is a fixed non-zero scaler parametrizing the theory). 
	\item Another example which exists in all dimensions is classical Dijkgraff-Witten theory. This theory, which in dimension $d$ is parametrized by a finite group $G$ and a characteristic class $\omega \in H^d(BG; \C^\times)$, assigns data to oriented manifolds equipped with principal $G$-bundles. It assigns trivial 1-dimensional vector spaces to each $(d-1)$-manifold and to a closed oriented $d$-manifold $M$ with principal $G$-bundle $P$ it assigns
\begin{equation*}
	\langle [M], \omega(P) \rangle
\end{equation*}	
	the $\omega$-characteristic number of $P$.
	\item An invertible Spin theory, a version of the Euler theory based on the Arf invariant, appears in Gunningham's work \cite{Gunningham:2012aa} on Spin Hurewicz numbers. 
	\item Similar theories give local or partially local formulas for many bordism invariants such as characteristic classes and the signature.  
	\item Invertible field theories govern and control anomalies in more general quantum field theories. See for example the work of Freed \cite{MR3330283}. 
	\item There are also recent real-world applications of invertible topologial field theories to condensed matter physics. Specifically the low energy behavior of gapped systems experiencing \emph{short-range entanglement} are well-modeled by invertible topological field theories, see for example \cite{Freed:2014aa, Kapustin:2015aa}.
	\item One approach to Quantum Chern-Simons theory describes it as an invertible 4-dimensional theory coupled together with a 3-dimensional boundary theory. See for instance \cite{MR2648901, walker-note-2006}
	\item Invertible field theories are also one of the key ingredients in the study of what are called `relative field theories' by Freed-Teleman \cite{MR3165462} and `twisted field theories' by Stolz-Teichner \cite{Stolz:2011aa}. 
\end{itemize}

For extended topological field theories and those valued in general targets we will say that a topological field theory is \emph{invertible} when it assigns invertible values to all manifolds and bordisms. In this case it takes values in an \emph{$\infty$-Picard} subcategory of the target. An $\infty$-Picard category is a symmetric monoidal $(\infty,n)$-category $E$ in which all objects and morphisms are invertible. It can also be defined as a symmetric monoidal $(\infty,n)$-category $E$ in which the shear map
\begin{equation*}
	(\otimes, \textrm{proj}_1): E \times E \to E \times E
\end{equation*}
is an equivalence. In this second definition it is clear that every object is $\otimes$-invertible, but in fact it also implies that every 1-morphism, 2-morphism, etc. is also invertible.  Hence $E$ is in actuality a symmetric monoidal $(\infty,0)$-category.

Another reason to single out the class of invertible topological field theories is that it is possible to completely classify them using stable homotopy theory. Grothendieck's homotopy hypothesis is the equivalence of homotopy theories between $(\infty,0)$-categories and topological spaces. This induces an equivalence between the homotopy theories of Picard $\infty$-categories and \emph{group-like} $E_\infty$-spaces, a.k.a.  connective spectra. Moreover by a well-known but unpublished result\footnote{An account of this is forthcoming \cite{Schommer-Pries-Invert}.}, which extends the work of Galatius-Madsen-Tillmann-Weiss \cite{MR2506750}, we can identify the geometric realization of $\Bord_d^{(X, \xi)}$ with the $E_\infty$-space $\Omega^{\infty-n} MT\xi$, a shift of a certain cover of the Madsen-Weiss spectrum. It then follows (see for example the discussion in \cite[Sect~2.5]{lurie-tft}) that extended field theories valued in the Picard $\infty$-category $E$ are in natural bijection with 
\begin{equation*}
	\pi_0\Map_{E_\infty}(\Omega^{\infty-n} MT\xi, E).
\end{equation*} 
That is with homotopy classes of  infinite loop maps from $\Omega^{\infty-n} MT\xi$ to $E$. In many cases, depending on $E$, this can be completely computed (see \cite{Freed:2014aa} for several examples where these computations are carried out).

\subsection{An Application: Crane-Yetter TQFTs are invertible}

The Crane-Yetter topological field theory \cite{MR1273569} is an oriented 4-dimensional field theory originally constructed from a modular tensor category. In \cite{MR1452438} this TQFT was shown to arise via a state-sum construction and the input was generalized to allow arbitrary balanced braided fusion categories, also called \emph{premodular categories}. 

This topological field theory is also known to be an extended field theory, as expected for any state-sum theory. Its description as an extended field theory has been given by Walker \cite{walker-note-2006} (see also \cite{Walker-ESI-talk}). It was also studied by Walker and Wang \cite{WalWan1104} (and is sometimes called the \emph{Walker-Wang model}). In that work they provide a skein-theoretic formula for the vector space associated to each 3-manifold. In the case of the 3-torus the vector space has a natural basis spanned by the indecomposable \emph{transperent objects}, those objects which braid trivially with all other objects\footnote{Specifically, $X$ is transparent if for each $Y$ we have $c_{X,Y} \circ c_{Y,X} = id_{Y \otimes X}$, where $c_{-,-}$ is the braiding morphism.}. 

Given a braided fusion category $\cC$ the subcategory of transparent objects will be a symmetric monoidal category which is called the \emph{M\"uger center} $\cZ_2(\cC)$. In the case of unitary categories, M\"uger showed \cite[Prop~2.11]{MR1749250} that a balanced braided fusion category is modular precisely if there is one irreducible transparent object, the unit object. This is known to hold also in the non-unitary case. 

Putting these facts together we see that the Crane-Yetter theory associated to a modular tensor category is an extended 4-dimensional theory in which the value associated to the 3-torus is a one-dimensional, hence invertible, vector space. It follows from our main theorem that the whole theory must then be invertible. Following the approach outlined above to classifying invertible theories using stable homotopy  we see that such theories are classified by
\begin{equation*}
	H_\otimes^0(MTSO(4), H\C^\times) \cong H^4(BSO(4); \C^\times) \cong \C^\times \times \C^\times
\end{equation*}
with the two factors corresponding to the Euler class and the $1^\text{st}$ Pontryagin class. Hence we recover the previously known result that the Crane-Yetter invariant is classical \cite{MR1273570,MR1362787}:

\begin{corollary*}
	If $\cC$ is a modular tensor category then there exist (non-zero) constants $\lambda_1, \lambda_2 \in \C$ such that the Crane-Yetter invariant of any closed oriented 4-manifold $W$ is given by
	\begin{equation*}
		CY(W) = \lambda_1^{\chi(W)} \cdot \lambda_2^{p_1(W)}
	\end{equation*} 
	where $\chi(W)$ and $p_1(W)$ are the Euler characteristic and $1^\text{st}$ Pontryagin number, respectively. 
\end{corollary*}

The numbers $\lambda_1$ and $\lambda_2$ are derived from the central charge and global dimension of the modular tensor category.

\subsection{Overview} Our main Theorem~\ref{thm:main} is a general result about extended field theories valued in symmetric monoidal $(\infty,n)$-categories. However in Section~\ref{sec:extending_down} we will show how the general case can be deduced from the case $n=2$, that is where $\Bord_d^{(X, \xi)}$ is a symmetric monoidal $(\infty,2)$-category.  We will call the corresponding field theories \emph{once-extended} topological field theories to indicate they have one additional categorical layer beyond the original Atiyah-Segal formulation.

Every symmetric monoidal $(\infty,2)$-category $\cC$ has an underlying symmetric monoidal weak 2-category $h_2\cC$, often called the \emph{homotopy 2-category} of $\cC$. This 2-category has the same objects and 1-morphisms as $\cC$, but the 2-morphisms are the equivalence classes of 2-morphisms in $\cC$.  The question of whether an $(\infty,2)$-categorical topological field theory 
\begin{equation*}
	\cZ: \Bord_d^{(X, \xi)} \to \cC
\end{equation*}
is invertible is completely determined by the corresponding  2-categorical theory:
\begin{equation*}
	h_2 \cZ: h_2\Bord_d^{(X, \xi)} \to h_2\cC.
\end{equation*}
$\cZ$ is invertible if and only if $h_2 \cZ$ is invertible. 

This permits us to eschew the world of $(\infty,2)$-categories and work entirely in the theory of symmetric monoidal bicategories. From now on, unless otherwise stated, $\Bord_d^{(X, \xi)}$ will denote the corresponding symmetric monoidal weak 2-category of bordisms, as constructed in \cite{Schommer-Pries:2011aa} (See \cite{scheimbauer-thesis} for a comparison between this notion and the $(\infty,2)$-categorical notion).   

After making these simplifications, our main theorem is proven inductively. One of the crucial tools which we use to compare theories of different dimensions is the technique of \emph{dimensional reduction}. Dimensional reduction is usually encountered in the context of oriented theories. If we fix a $k$-dimensional oriented manifold $M$, then this gives rise to a symmetric monoidal functor
\begin{equation*}
	(-) \times M: \Bord^{SO}_{d-k} \to \Bord^{SO}_d.
\end{equation*}
which sends a manifold $Y$ to $Y \times M$. 
If we are given a $d$-dimensional field theory $\cZ$, then by pre-composing with the above map we obtain a $(d-k)$-dimensional theory. Thus theorems about lower dimensional theories have direct consequences for higher dimensional theories as well. 

In the presence of general tangential structures, the process of dimensional reduction is much more subtle. More importantly both our induction strategy and the base case $(d=2)$ become significantly more complicated. For this reason we will first concentrate on the oriented version of the main theorem, and then explain how to adapt the argument in the presence of general tangential structures. 

We have organized this paper as follows: 
\begin{itemize}
	\item  In Section~\ref{sec:cat_prelim} we establish a few fundamental algebraic/categorical results about detecting invertibility in monoidal categories.
	\item In Section~\ref{sec:moving_up} we prove the key Lemma~\ref{lem:moving_up} which shows that if the bottom two layers of a once-extended topological field theory are invertible, then so is the top layer (invertibility can be `pushed upward').
	\item In Section~\ref{sec:base_case_SO} we prove the base case ($d=2$) of our theorem for oriented theories. The argument in this section is due to Freed-Teleman \cite{Freed-Aspect}.  
	\item In Section~\ref{sec:induct_SO} we give the inductive argument establishing our main theorem, again only in the oriented setting.
	\item In Section~\ref{sec:general_tangential} we discuss general tangential structures.
	\item  In Section~\ref{sec:moving_up_general} we adapt the results of Section~\ref{sec:moving_up} to the case of general tangential structures. 
	\item In Section~\ref{sec:base_genstr} we discuss two-dimensional theories with general structures. We review some basic facts about 2-framed bordisms, we introduce the notion of \emph{spherophilic} tangential structures (those such that the 2-sphere admits such a structure), and we prove the base case ($d=2$) when the tangential structure is  spherophilic.
	\item In Section~\ref{sec:dimensional_redux} we describe the various forms of dimensional reduction that we will need to use in the presence of general tangential structures.  
	\item In Section~\ref{sec:induct_general} we describe how to modify our previous inductive argument (given for oriented theories) to the case of general tangential structures. 
	\item In Section~\ref{sec:extending_down} we show how to use the $(\infty,2)$-categorical results already established to obtain the $(\infty,n)$-categorical results stated in Theorem~\ref{thm:main}.
\end{itemize}

% spinnable 

%- all manifolds will be compact. 

\subsection*{Acknowledgements}

I am extremely grateful to Chao-Ming Jian whose question \cite{MO-218982} on MathOverflow was an important source of inspiration for this work. I am also grateful to Dan Freed and Constantin Teleman for their work on detecting invertibility in fully-local extended field theories and for the key calculation in Section~\ref{sec:base_case_SO} of oriented version of the base case, $d=2$. And, for their immense hospitality, I am indebted to the Max-Planck Institute for Mathematics in Bonn, where this work was carried out.

\section{Categorical observations} \label{sec:cat_prelim}
The following lemma will be used repeatedly. 

\begin{lemma}\label{lem:symmetric_mon_invert_mor}
	Let $\cC$ be a monoidal category. Let $f: x \to y$ and $g:y \to z$ be morphisms in $\cC$. Suppose that the objects $x$, $y$, and $z$ are invertible in $\cC$, and that the composite $g \circ f$ is an  isomorphism. Then both $f$ and $g$ are  isomorphisms. 
\end{lemma}

\begin{proof}
	Composition with isomorphisms preserves and reflects the property of being an isomorphism. Thus, by post-composing $g$ with $(gf)^{-1}$, we may assume without loss of generality that $z = x$ and that $g \circ f = id_x$. In other words $f$ and $g$ exhibit $x$ as a retract of $y$. 
	
	Each object $w \in \cC$ gives rise to an endofunctor $w \otimes (-): \cC \to \cC$, and if $w$ is invertible, then this is an equivalence of categories; it then reflects isomorphisms. It follows, by tensoring with the inverse of $y$ and composing with the isomorphism $y \otimes y^{-1} \cong 1$,  that we may assume  that $y =1$, the unit object.
	
In short we have reduced to the case that that $f$ and $g$ exhibit $x$ as a retract of the unit object. 
The morphisms $f \otimes f$ and $g \otimes g$ exhibit $x \otimes x$ as a retract of $1 \otimes 1 \cong 1$, and the following diagram:
\begin{center}
\begin{tikzpicture}
		\node (LT) at (0, 1.5) {$ 1  $};
		\node (LB) at (0, 0) {$ x \cong x \otimes 1 $};
		\node (RT) at (3, 1.5) {$ 1 \otimes 1 $};
		\node (RB) at (3, 0) {$ x \otimes x $};
		\draw [->] (LT) -- node [right] {$ g $} (LB);
		\draw [->] (LT) -- node [above] {$ \cong $} (RT);
		\draw [->] (RT) -- node [left] {$ g \otimes g $} (RB);
		\draw [->] (LB) -- node [below] {$ id \otimes g $} (RB);
		\draw [->, dashed] (LB) to [out = 135, in = 225] node [left] {$f$} (LT);
		\draw [->, dashed] (RB) to [out = 45, in = -45] node [right] {$f \otimes f$} (RT);	
		%\node at (0.5, 1) {$\ulcorner$};
		%\node at (1.5, 0.5) {$\lrcorner$};
\end{tikzpicture}
\end{center} 	
then exhibits the arrow $(Id \otimes g): x \cong x \otimes 1 \to x \otimes x $ as  a retract of the isomorphism $1 \cong 1 \otimes 1$. Since retracts of isomorphisms are isomorphisms, we conclude that $(Id \otimes g)$ gives an isomorphism $x \cong x \otimes x$. However since $x$ is invertible, we may cancel the left-hand copy of $x$ and conclude that $g: 1 \to x$ is an isomorphism. It follows that both $f$ and $g$ are isomorphisms, as desired.  	
\end{proof}

\begin{remark}
	In fact this lemma can be generalized a bit further. It also holds when $x$, $y$, and $z$ are invertible parallel morphisms in a weak 2-category. 
\end{remark}

\noindent Recall the following standard fact:

\begin{lemma}\label{lem:invertible_unit}
	A 1-morphism $f$ in a 2-category is invertible if and only if it admits an adjoint such that the unit and counit 2-morphisms of the adjunction are invertible. \qed
\end{lemma}

\section{First observations about invertible field theories} \label{sec:moving_up}
Our first result shows that that invertibility can be pushed upward, provided there are at least two consecutive layers of invertibility. 

\begin{lemma}\label{lem:moving_up}
	Let $\cZ: \Bord^{SO}_d \to \cC$ be a once-extended topological field theory such that $\cZ(Y)$ is invertible for each $(d-2)$-dimensional manifold $Y$, and $\cZ(\Sigma)$ is invertible for every $(d-1)$-dimensional bordism $\Sigma$. Then $\cZ$ is invertible. 
\end{lemma}

\begin{proof}
	We must show that every $d$-dimensional bordism with corners is assigned an invertible value under $\cZ$. For this end let $M_0$ and $M_1$ be closed $(d-2)$-manifolds, let $\Sigma_0$ and $\Sigma_1$ be $(d-1)$-dimensional bordisms from $M_0$ to $M_1$, and let $W$ be a $d$-dimensional bordism with corners from $\Sigma_0$ to $\Sigma_1$. 
		We may choose a Morse function $f:W \to [0,1]$ with critical points only on the interior of $W$, having distinct critical values, and with $f^{-1}(i) = \Sigma_i$ for $i=0,1$. By slicing in between the critical values, we observe that $W$ is a composite of $d$-dimensional bordisms with corners which have only a single critical point. Thus it is sufficient to prove that such bordism are invertible. 
		
	Every such bordism $W$ (a $d$-dimensional bordism with corners which has only a single Morse critical point) is given by a `handle attachment'. Specifically let $p+q = d-2$, with $p,q \geq -1$.  We may regard $S^p \times D^{q+1}$ and $D^{p+1} \times S^q$ as $(d-1)$-dimensional cobordisms from $S^p \times S^q$ to the empty manifold $\emptyset$. We may also regard $D^{p+1} \times D^{q+1}$ as a $d$-dimensional bordism with corners from $S^p \times D^{q+1}$ to $D^{p+1} \times S^q$.  
	
	For every $d$-dimensional bordism with corners with a single Morse critical point there exists a pair of integers $p$ and $q$ with $p+q = d-2$, $p,q \geq -1$, and a $(d-1)$-dimensional bordism $\Sigma$ from $M_0$ to $M_1 \sqcup S^p \times S^q$, such that:
	\begin{align*}
		\Sigma_0 &= (M_1 \times I \sqcup S^p \times D^{q+1}) \circ \Sigma \\
		\Sigma_1 &= (M_1 \times I \sqcup D^{p+1} \times S^q) \circ \Sigma \\
		W &= (M_1 \times I^2 \sqcup D^{p+1} \times D^{q+1})\circ \Sigma \times I
	\end{align*}
	This decomposition is possible because our bordism category has codimension-two corners, the same as handles. From this decomposition we see that it is sufficient to show that each handle bordism $D^{p+1} \times D^{q+1}$ is assigned an invertible value. 
		
However handles may be canceled. Specifically consider an index $(p-1)$-handle: 
\begin{equation*}
	D^{p} \times D^{q+2}: S^{p-1} \times D^{q+2} \to D^p \times S^{q+1}
\end{equation*}
We may `whisker' this with the bordism $D^p \times S^{q+1}$, thought of as a morphism from $\emptyset$ to $S^{p-1} \times S^{q+1}$. The result is a bordism
\begin{equation*}
	H_{p-1}: S^{d-1} \cong D^p \times S^{q+1} \cup_{S^{p-1} \times S^{q+1}} S^{p-1} \times D^{q+2} \to  D^p \times S^{q+1} \cup_{S^{p-1} \times S^{q+1}} D^p \times S^{q+1} \cong S^p \times S^{q+1}.
\end{equation*}
Since $D^p \times S^{q+1}$ is assigned an invertible 1-morphism, the $(p-1)$-handle is assigned an invertible 2-morphism if and only if $H_{p-1}$ is assigned an invertible 2-morphism. 

Similarly consider the $p$-handle: 
\begin{equation*}
	D^{p+1} \times D^{q+1}: S^p \times D^{q+1} \to D^{p+1} \times S^q.
\end{equation*}
We may `whisker' this with the bordism $S^p \times D^{q+1}$, thought of as a bordism from $\emptyset$ to $S^p \times S^q$ to obtain a bordism:
\begin{equation*}
	H_p: S^p \times S^{q+1} \cong (S^p \times D^{q+1}) \cup_{S^p \times S^q} S^p \times D^{q+1} \to (S^p \times D^{q+1}) \cup_{S^p \times S^q} D^{p+1} \times S^q \cong S^{d-1}.
\end{equation*}
Again the $p$-handle will be assigned an invertible value if an only if $H_p$ is assigned an invertible value. 

But the composite $H_p \circ H_{p-1}$ is the identity bordism of $S^{d-1}$ and since the source and targets of both $H_p$ and $H_{p-1}$ are assigned invertible values (by assumption) it follows from Lemma~\ref{lem:symmetric_mon_invert_mor} that both $H_p$ and $H_{p-1}$ are assigned invertible values. 
\end{proof}

\noindent This has an important corollary for fully-local field theories. 

\begin{corollary}\label{cor:invertblepointsareinvertible}
	In any oriented fully extended topological field theory, if the values of all zero-manifolds are invertible, then the field theory is invertible.  
\end{corollary}

\begin{proof}
	The value of $-1$-dimensional manifolds (i.e. the empty manifold) and $0$-dimensional manifolds are invertible. Thus applying the previous lemma, we see that the value of each $1$-dimensional bordism is invertible. Applying the lemma again, one dimension higher, we see that the value of each $2$-dimensional bordism is invertible. Continuing in this way shows that every bordism is assigned an invertible value. 
\end{proof}

\section{The base case, oriented version} \label{sec:base_case_SO}
We will now establish a very special case of our main theorem. We will consider oriented 2-dimensional extended topological field theories $\cZ$ and show that if the value $\cZ(S^1)$ of the circle is invertible, then the entire field theory is invertible. For very particular target categories, this is an easy consequence of the classification in \cite{Schommer-Pries:2011aa}. However the following proof for oriented field theories with general target categories is due to Dan Freed and Constantin Teleman. 
We learned of it from Dan Freed's Aspect lecture \cite{Freed-Aspect}. In Section~\ref{sec:base_genstr} we will adapt these results to a slightly larger class of tangential structures.

\begin{figure}[htbp]
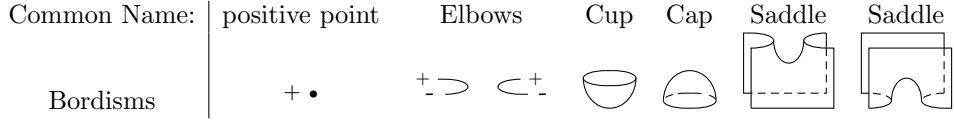

	\centering
		%\includegraphics[scale=1]{file}
		%Include a picture with the following bordims:
%\begin{itemize}
%	\item Points
%	\item Elbows
%	\item Caps
%	\item Saddles
%	\item pants
%\end{itemize}
		\begin{tabular}{c|cccccc}
			Common Name: & positive point & Elbows & Cup & Cap & Saddle & Saddle \\
			Bordisms & 
				\tikz{	\node at (-4, 1.5) {\footnotesize + {\tiny $\bullet$}};} &
				\tikz{
				\draw (-6.6, -.5) -- (-6.5, -.5) arc (90: -90: 0.3cm and 0.1cm);
				\draw (-5.5, -.5) arc (90: 270: 0.3cm and 0.1cm) -- (-5.4, -.7);
				\node at (-6.8, -.5) {\tiny +};
				\node at (-6.7, -.7) {-};
				\node at (-5.3, -.5) {\tiny +};
				\node at (-5.2, -.7) {-};
				} & 
				\tikz{	\draw (.1, 3) arc (90: -90: 0.3cm and 0.1cm) -- (0, 2.8);
	\draw (0, 3) arc (90: 270: 0.3cm and 0.1cm) ;
	\draw (0, 3) -- (.1, 3);
	\draw (-.3, 2.9) arc (180: 360: 0.35cm and 0.4cm);} & 
				\tikz{	\draw (.1, 2) -- (0,2);
	\draw (0,2) arc (270: 180: 0.3cm and 0.1cm);
	\draw [densely dashed] (0, 2.2) arc (90: 180: 0.3cm and 0.1cm); 
	\draw [densely dashed] (.1, 2.2) -- (0, 2.2);
	\draw (.1,2) arc (-90: 0: 0.3cm and 0.1cm);
	\draw [dashed] (0.1, 2.2) arc (90:0: 0.3cm and 0.1cm);
	\draw (-.3, 2.1) arc (180: 0: 0.35cm and 0.4cm);} &
				\tikz{	\draw (0, 3) arc (90: -90: 0.3cm and 0.1cm) -- (0, 2);
	\draw (1, 3) arc (90: 270: 0.3cm and 0.1cm) -- (1.1, 2.8) -- (1.1, 2) -- (0, 2);
	\draw (1,3) -- (1, 2.8); \draw [densely dashed] (1, 2.8) -- (1, 2.2) -- (0,2.2);
	\draw (.3, 2.9) arc (180: 360: 0.2cm and 0.3cm);
	\draw (0, 3) -- (-.1, 3) -- (-.1, 2.2) -- (0, 2.2); } &
				\tikz{\draw (2,2) arc (-90: 0: 0.3cm and 0.1cm);
	\draw [densely dashed] (2.3, 2.1) arc (0:90: 0.3cm and 0.1cm);
	\draw (2, 2.2) -- (1.9, 2.2) -- (1.9, 3) -- (3,3) -- (3, 2.8);
	\draw [densely dashed] (3, 2.8) -- (3, 2.2);
	\draw (2,2) -- (2, 2.8) -- (3.1, 2.8) -- (3.1, 2) -- (3,2);
	\draw (3,2) arc (270: 180: 0.3cm and 0.1cm);
	\draw [densely dashed] (3, 2.2) arc (90: 180: 0.3cm and 0.1cm); 
	\draw (2.3, 2.1) arc (180: 0: 0.2cm and 0.3cm);} 
		\end{tabular}
		
	\caption{Some 2-dimensional bordisms}
	\label{fig:2Dbordisms}
\end{figure}

Figure~\ref{fig:2Dbordisms} depicts some important 2-dimensional bordisms. Each point is given an orientation after stabilizing its (null) tangent bundle to the trivial rank-2 bundle. Thus, up to isomorphism, there are precisely two oriented points: the positive point and the negative point. These objects are dual in $\Bord^{SO}_2$, and the unit and counit of the adjunction between them are given by the `elbow' bordisms\footnote{actually we must compose one of the elbows by the `swap' bordism, which is the unique invertible bordism from $pt^+ \sqcup pt^-$ to $pt^- \sqcup pt^+$.}, depicted in Figure~\ref{fig:2Dbordisms}. The elbows, in turn, are also adjoint to each other. In fact they are ambidexterously adjoint (form both a left and right adjunction). The cup, cap, and saddles provide the units and counits for these adjunctions.

\begin{proposition}\label{pro:basecase}
	Let $\cZ: \Bord_2^{SO} \to \cC$ be an oriented extended topological field theory. Then $\cZ$ is invertible if and only if $\cZ(S^1)$ is invertible.  
\end{proposition}

\begin{proof}
	The value $\cZ(S^1)$ is invertible in any invertible field theory; the more important implication is the converse. So we assume that $\cZ(S^1)$ is invertible. By Cor.~\ref{cor:invertblepointsareinvertible} it is sufficient to show that the value assigned to every zero manifold is invertible. Every zero manifold is a disjoint union of positive and negative points, which are dual to each other. Thus by Lemma~\ref{lem:invertible_unit} it is enough to show that the unit and counit of the duality between the positive and negative point, that is the elbow bordisms, are assigned invertible values. Up to composing with an invertible `swap' bordism, the elbow bordisms are adjoint, so again by Lemma~\ref{lem:invertible_unit} it is enough to show that the unit and counit of this adjunction, that is the cup and the saddle bordisms, are assigned invertible morphisms. 
	
We will first show that the cup bordism is assigned an invertible value. Observe that the circle is a self-dual object in $\Bord_2^{SO}$. The unit and counit of this self-duality are given by the annulus $S^1 \times [0,1]$, which can be read either as a bordism from the empty $1$-manifold to $S^1 \sqcup S^1$, or as a bordism the other way around.  Since $\cZ(S^1)$ is invertible, it follows from Lemma~\ref{lem:invertible_unit} that the values of each of these annuli are invertible. 

Each annulus can be written as a composite of a `pair-of-pants' bordism and a disk (cup or cap). By Lemma~\ref{lem:symmetric_mon_invert_mor} and the fact that $\cZ(S^1 \sqcup S^1) = \cZ(S^1) \otimes \cZ(S^1)$ is invertible, it follows that the values of both the disk and the pants are invertible, regardless of which direction these bordisms are read. Every 2-dimensional bordism between closed 1-manifolds can be obtained as composites of these, and hence the value of $\cZ$ on any  2-dimensional bordism between closed 1-manifolds is invertible. Note that a special case of this is the cup bordism. 

Now we will show that the saddle bordism is assigned an invertible value. We will show that the reverse saddle gives a two-sided inverse to the saddle, after applying the given field theory $\cZ$. The calculations are depicted in Figure~\ref{fig:decomposition_coposedSaddles1} and Figure~\ref{fig:decomposition_coposedSaddles2}, where these pictures are meant to be taking place in $\cC$ (after applying $\cZ$); The key step in both is the next to last, where we apply the identity depicted in Figure~\ref{fig:Cylinder-Sphere-eqn}.\end{proof}

\begin{figure}[htbp]
	\centering
		\includegraphics[scale=.24]{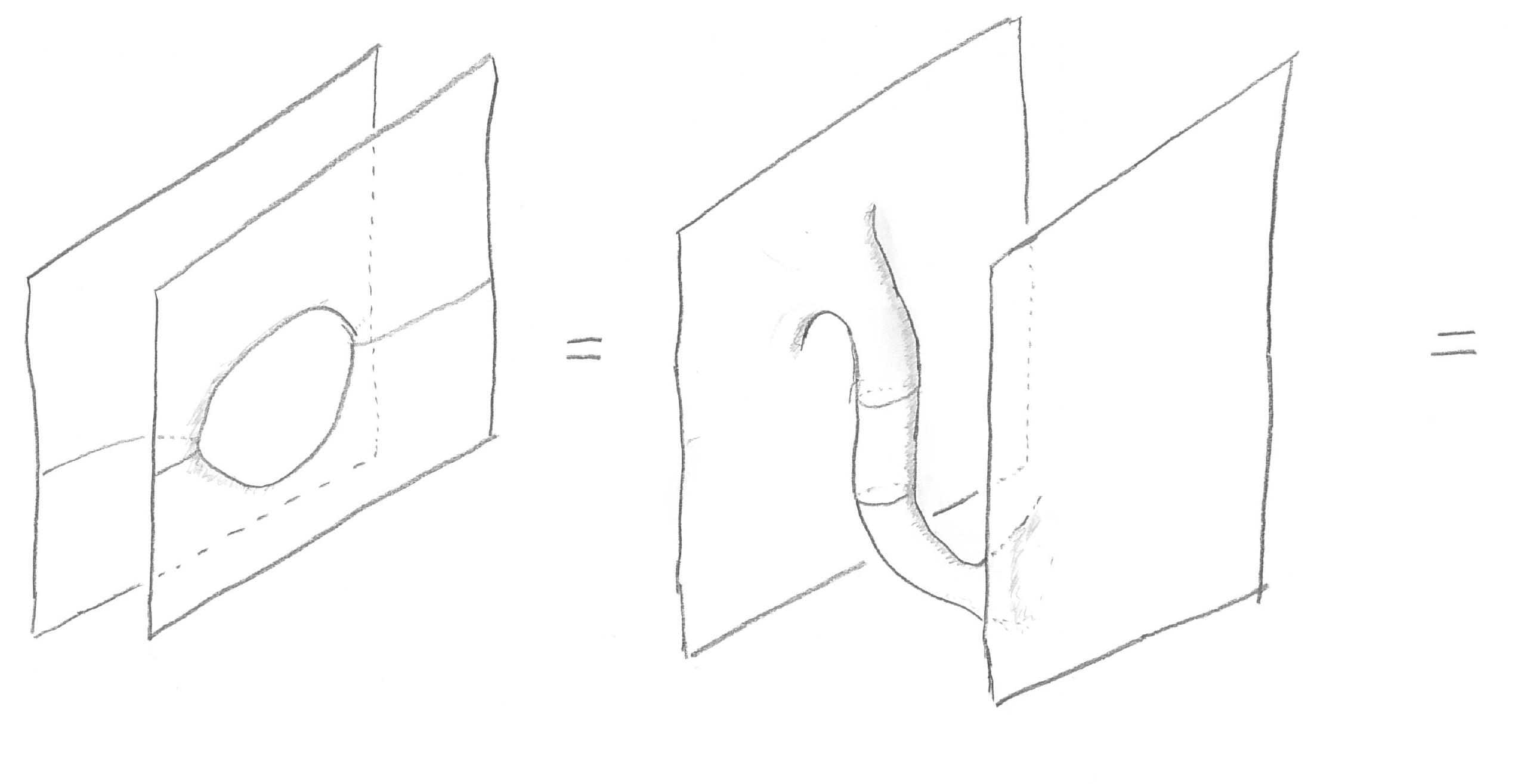}
		\includegraphics[scale=.22]{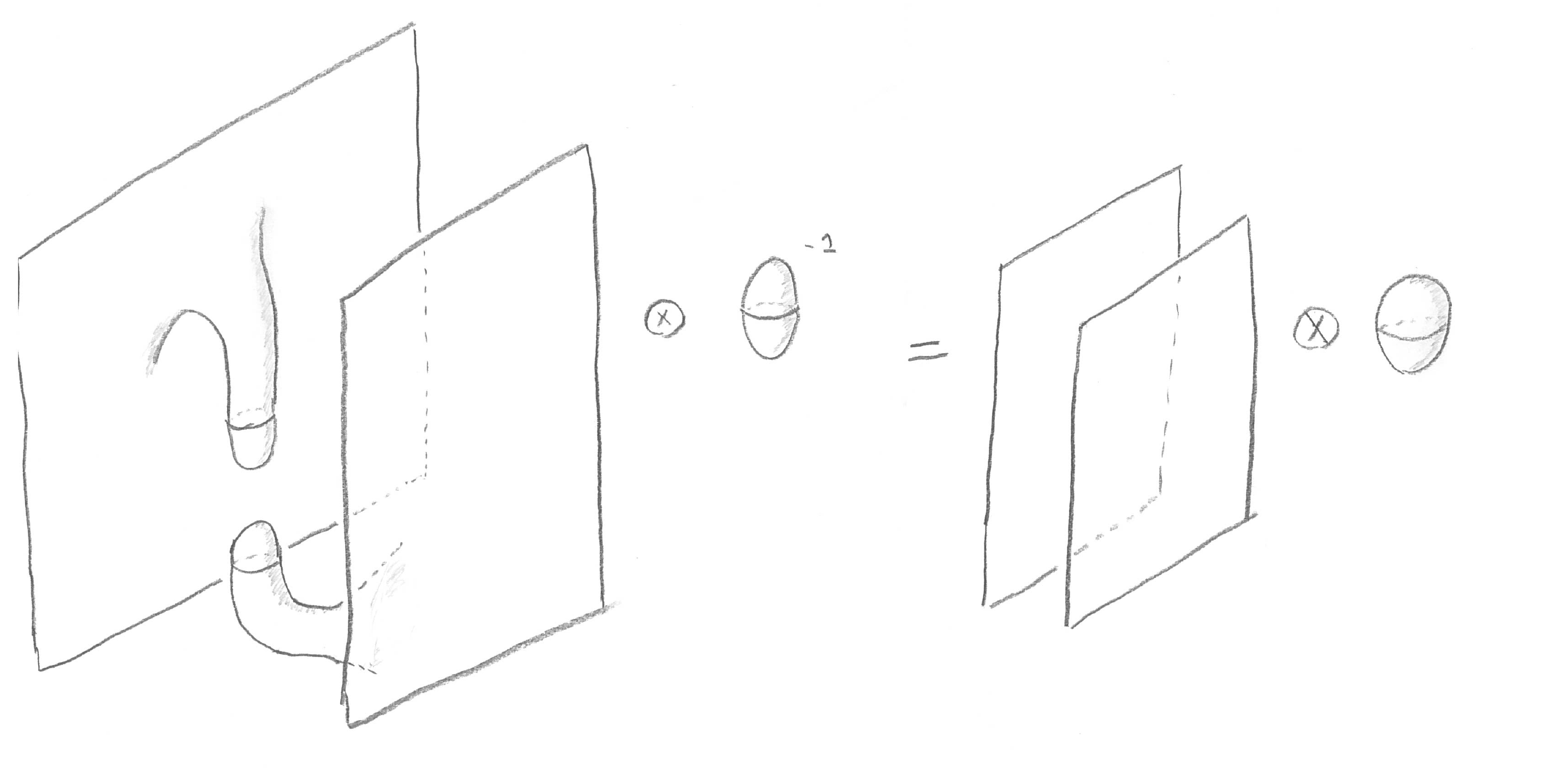}
	\caption{Invertiblity of the Saddle, one way.}
	\label{fig:decomposition_coposedSaddles1}
\end{figure}

\begin{figure}[htbp]
	\centering
		\includegraphics[scale=.4]{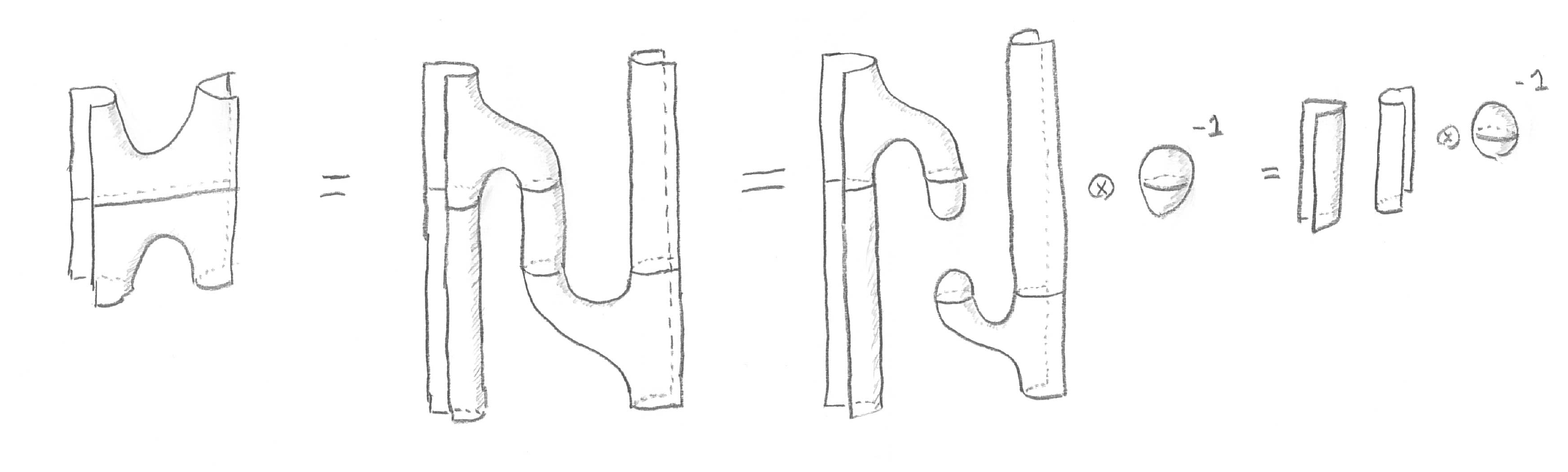}
	\caption{Invertiblity of the Saddle other way.}
	\label{fig:decomposition_coposedSaddles2}
\end{figure}

\begin{figure}[htbp]
	\centering
		\includegraphics[scale=.4]{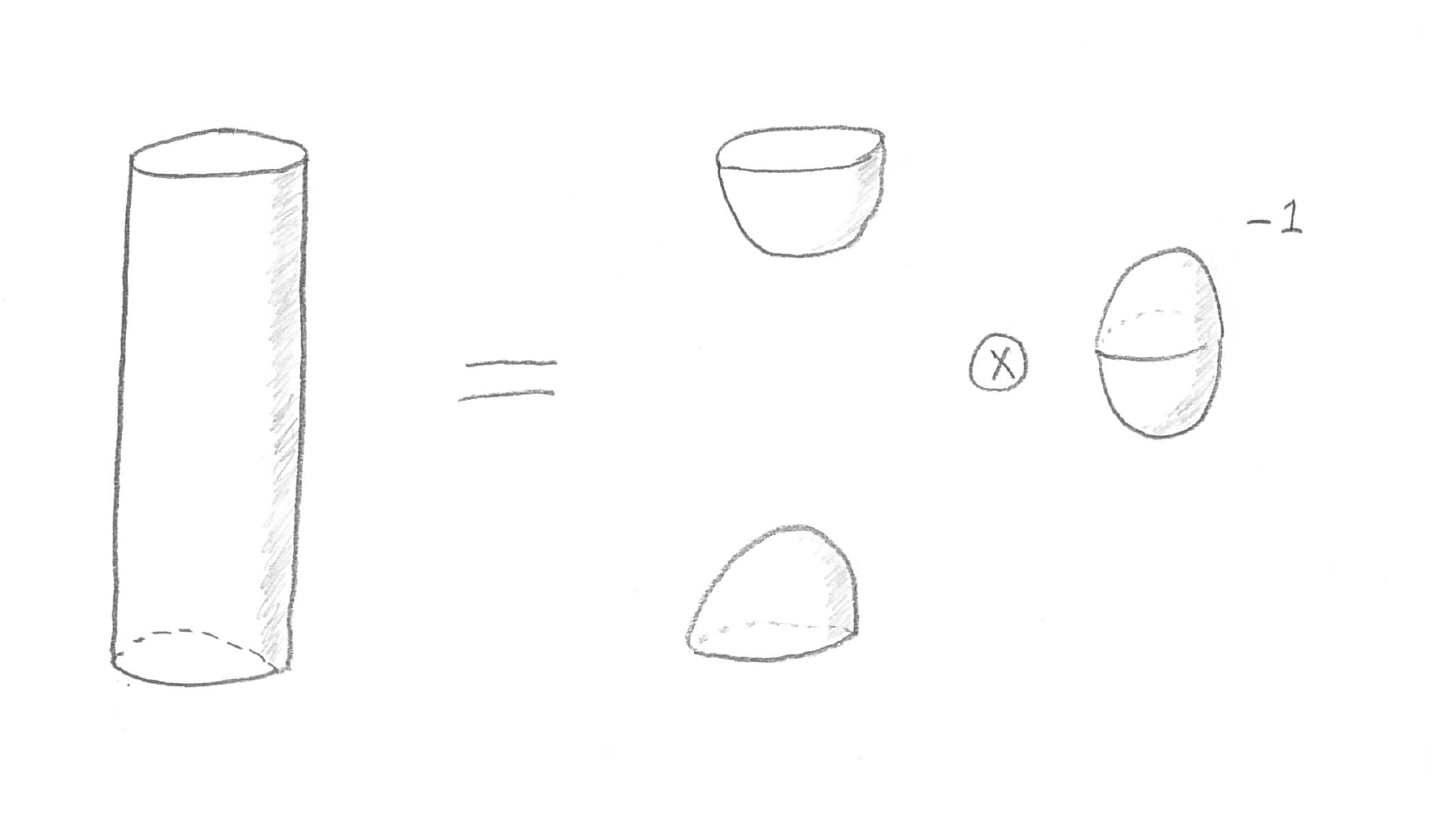}
	\caption{The cylinder in terms of a cup, cap, and inverse sphere.}
	\label{fig:Cylinder-Sphere-eqn}
\end{figure}

\section{The main theorem, oriented version} \label{sec:induct_SO}
We will now prove our main theorem in the oriented case. 

\begin{theorem}\label{thm:main_thm_SO}
	Let $\cZ: \Bord_d^{SO} \to \cC$ be a once-extended topological field theory.
	%, valued in the symmetric monoidal 2-category $\cC$. 
	Then $\cZ$ is invertible if and only if the value of the $(d-1)$-torus $\cZ(T^{n-1})$ is invertible. 
\end{theorem}

\begin{proof}
We will prove this theorem by induction. The base case $d=2$ is Prop.~\ref{pro:basecase}. Thus we may assume, that the theorem statement holds for all dimensions $p < d$. We first consider the effect of dimensional reduction along the circle, i.e. precomposition with 
\begin{equation*}
	(-) \times S^1: \Bord_{d-1}^{SO} \to \Bord_d^{SO}
\end{equation*}
Let $\cZ_{S^1}$ denote the dimensionally reduced theory. Thus $\cZ_{S^1}(X) = \cZ(X \times S^1)$. As a $(d-1)$-dimensional theory we have that $\cZ_{S^1}(T^{d-2}) = \cZ(T^{d-2} \times S^1 ) = \cZ(T^{d-1})$ is invertible. Thus, by our induction hypothesis, the entire theory $\cZ_{S^1}$ is an invertible theory. It follows that for every closed oriented $(d-2)$-manifold $M$, we have $\cZ_{S^1}(M) = \cZ(M \times S^1)$ is invertible. 

Now we will consider each oriented $(d-2)$-manifold $M$ separately and contemplate dimensional reduction along $M$, i.e. precomposition with 
\begin{equation*}
	M \times (-): \Bord_{2}^{SO} \to \Bord_d^{SO}.
\end{equation*}
Let $\cZ_M$ denote this dimensionally reduced theory. Since $\cZ_M(S^1) = \cZ(M \times S^1)$ is invertible, by Prop.~\ref{pro:basecase} the whole theory $\cZ_M$ is invertible. Hence $\cZ_M(pt) = \cZ(M)$ is invertible. In particular we have now shown that the value of $\cZ$ on every $(d-2)$-manifold is invertible. 

We will now consider when the value of $\cZ$ on closed $(d-1)$-manifolds is invertible. We first establish a lemma:

\begin{lemma}\label{lem:surgery_invertibility}
	Let $\cZ: \Bord_d^{SO} \to \cC$ be a once-extended topological field theory such that for each $(d-2)$-manifold $M$, $\cZ(M)$ is invertible. Suppose that the closed $(d-1)$-manifold $N_2$ is obtained from $N_1$ by surgery along an embedded $S^k \times D^{d-1-k}$. Suppose that $N_1$ is non-empty and $\cZ(N_1)$  is invertible, then $\cZ(N_2)$ is invertible. 
\end{lemma}

	Consider $N_1 \setminus S^k \times D^{d-1-k}$ as a $(d-1)$-dimensional bordism from $\emptyset$ to $S^k \times S^{d-2-k}$. Let $H_1$ be $S^k \times D^{d-1-k}$ viewed as a bordism from $S^k \times S^{d-2-k}$ to $\emptyset$. Then the composition yields:
	\begin{equation*}
		H_1 \circ (N_1 \setminus S^k \times D^{d-1-k}) \cong N_1
	\end{equation*}
	and hence 
	\begin{equation*}
		 \cZ(H_1) \circ  \cZ(N_1 \setminus S^k \times D^{d-1-k}) \cong  \cZ(N_1).
	\end{equation*}
	Since $\cZ(N_1)$ is invertible and the sources and targets of these morphisms are invertible by assumption, it follows from Lemma~\ref{lem:symmetric_mon_invert_mor}, that both $\cZ(H_1)$ and $ \cZ(N_1 \setminus S^k \times D^{d-1-k})$ are invertible. 
		
	Now let $H_2$ be $D^{k+1} \times S^{d-2-k}$ viewed as a bordism from $S^k \times S^{d-2-k}$ to $\emptyset$. Since $N_1$ is non-empty, we can choose an embedding of $D^{k+1} \times S^{d-2-k}$ into $N_1$ (for example such an embedding exists in a sufficiently small ball). The same argument as above shows that $\cZ(H_2)$ is invertible as well. But now we have (by the definition of surgery)
	\begin{equation*}
		N_2 \cong H_2 \circ (N_1 \setminus S^k \times D^{d-1-k})
	\end{equation*}
	which implies that $\cZ(N_2)$ is a composite of two invertible morphisms, hence invertible. This establishes the lemma.  
	
Returning to the proof of Theorem~\ref{thm:main_thm_SO}, we know that $\cZ(T^{d-1})$ is invertible by assumption. By repeatedly applying the above Lemma~\ref{lem:surgery_invertibility}, this implies that if $N$ is a closed $(d-1)$-manifold which can be obtained from $T^{d-1}$ by repeated surgeries (i.e. such that $N$ is bordant to $T^{d-1}$) then $\cZ(N)$ is invertible. 

In general not every closed $(d-1)$-manifold $N$ is bordant to the $(d-1)$-torus (which itself is null-bordant). The obstruction to this is the class $[N] \in \Omega^{SO}_{d-1}$ in the oriented bordism group. However if $N$ is any oriented closed $(d-1)$-manifold, and $\overline{N}$ denotes the orientation reversal of $N$, then $[N \sqcup \overline{N}] = [N] - [N] = 0 \in \Omega^{SO}_{d-1}$, and hence $N \sqcup \overline N$ may be obtained from $T^{d-1}$ by repeated surgeries. It then follows from Lemma~\ref{lem:surgery_invertibility} that
\begin{equation*}
	\cZ(N \sqcup \overline N) \cong \cZ(N) \otimes \cZ(\overline{N}) \cong \cZ(\overline{N}) \otimes  \cZ(N)
\end{equation*}
is invertible, and hence both $\cZ(N)$ and $\cZ(\overline{N})$ are invertible. So we have just established that the value of $\cZ$ on every closed $(d-1)$-manifold is invertible. 

Next we note that if $W$ is a $(d-1)$-dimensional bordism such that either the source or the target is the empty manifold $\emptyset$, then $\cZ(W)$ is invertible. We see this as follows. Without loss of generality assume the source of $W$ is the empty manifold, and the target is $M$. Then $\overline{W} \cup_M W$ is a closed $(d-1)$-manifold and hence
\begin{equation*}
	\cZ(\overline{W} \cup_M W) \cong \cZ(\overline{W}) \circ \cZ(W)
\end{equation*}
 is invertible. Since the sources and targets of these morphisms are invertible, it again follows from Lemma~\ref{lem:symmetric_mon_invert_mor} that both $\cZ(\overline{W})$ and $\cZ(W)$	are invertible. 
	
Finally we will consider an arbitrary $(d-1)$-dimensional bordism $W$ from $M_1$ to $M_2$. Since $\cZ(\overline{M_1})$ is an invertible object, tensoring with $\cZ(\overline{M_1})$ is an equivalence, and hence preserves and reflects invertibility.  It follows that $\cZ(W)$ is invertible if and only if 
\begin{equation*}
	id_{\cZ(\overline{M_1})} \otimes \cZ(W) \cong \cZ(\overline{M_1} \times I) \otimes \cZ(W) \cong \cZ(\overline{M_1} \times I \sqcup W)
\end{equation*}
is invertible, where $\overline{M_1} \times I$ is the identity bordism of $\overline{M_1}$. But we can also view $\overline{M_1} \times I$ as a bordism from $\emptyset$ to $\overline{M_1} \sqcup M_1$, which we will denote simply $X$, to distinguish from the identity. Both $X$ and the composite 
\begin{equation*}
	(\overline{M_1} \times I \sqcup W) \circ X
\end{equation*}
are bordism with source the empty manifold $\emptyset$. Hence their values under $\cZ$ are invertible, and it follows that $\cZ(\overline{M_1} \times I \sqcup W)$, and hence also $\cZ(W)$, is invertible. 

So at last we have established that the value of $\cZ$ on every $(d-2)$-manifold and every $(d-1)$-dimensional bordism is invertible. The theorem now follows directly from Lemma~\ref{lem:moving_up}. 
\end{proof}

\section{Tangential structures} \label{sec:general_tangential}

Careful constructions of the bordism  higher category \cite{MR2764873, scheimbauer-thesis} construct it using bordisms embedded into a large or infinite dimensional Euclidean space. In particular, for the $(\infty,n)$-categorical version, the top dimensional $d$-manifolds will be embedded into $\R^n \times \R^\infty$. The lower dimensional manifolds corresponding to lower dimensional morphisms of $\Bord_d$, will be embedded into $\R^k \times \R^\infty$ with $k \leq n$. 

Thus in the remainder of this paper we will tacitly assume that all our manifolds are embedded into these spaces as well. We will use the Grassmannian of $d$-planes $Gr_d(\R^n \times \R^\infty)$ as our model for the classifying space $BO(d)$. In particular since all of our $d$-manifolds are embedded in $\R^n \times \R^\infty$ the tangent bundle yields a  Gauss map
\begin{equation*}
	\tau_M: M \to Gr_d(\R^n \times \R^\infty) \simeq BO(d)
\end{equation*}
which is well-defined at the point-set level, not just up to homotopy. Moreover, our lower dimensional manifolds (say of dimension $d-k$) have well-defined stabilization Gauss maps
\begin{equation*}
	(\tau_y \oplus \varepsilon^k): Y \to Gr_d(\R^n \times \R^\infty) \simeq BO(d).
\end{equation*}
These maps will be used to define tangential structures at the point-set level. 

A space $\xi: X \to BO(d)$ over the classifying space $BO(d)$ determines a type of tangential structure for manifolds of dimension $\leq d$. If $\xi:X \to BO(d)$ is a fibration than a tangential structure for $M$ is given by a lift $\theta$
\begin{center}
\begin{tikzpicture}
		%\node (LT) at (0, 1.5) {$$};
		\node (LB) at (0, 0) {$M$};
		\node (RT) at (2, 1.5) {$X$};
		\node (RB) at (2, 0) {$BO(d)$};
		%\draw [->] (LT) -- node [left] {$$} (LB);
		\draw [->, dashed] (LB) -- node [above left] {$\theta$} (RT);
		\draw [->>] (RT) -- node [right] {$\xi$} (RB);
		\draw [->] (LB) -- node [below] {$ $} (RB);
		%\node at (0.5, 1) {$\ulcorner$};
		%\node at (1.5, 0.5) {$\lrcorner$};
\end{tikzpicture}
\end{center}
Tangential structures for lower dimensional manifolds are defined the same way, using the stabilized Gauss map. 

The natural notion of equivalence of $(X,\xi)$-structure is that of \emph{isotopy}, which means homotopy over the space $BO(d)$. We will write $\theta_0 \simeq \theta_1$ when $\theta_0$ is isotopic to $\theta_1$. Manifolds with isotopic $(X,\xi)$-structures are equivalent in $\Bord_d^{(X,\xi)}$. Moreover a commutative triangle
\begin{center}
\begin{tikzpicture}
		\node (LT) at (0, 1.5) {$X$};
		\node (LB) at (1, 0) {$BO(d)$};
		\node (RT) at (2, 1.5) {$X'$};
		%\node (RB) at (2, 0) {$ $};
		\draw [->] (LT) -- node [left] {$\xi$} (LB);
		\draw [->] (LT) -- node [above] {$\sim$} (RT);
		\draw [->] (RT) -- node [right] {$\xi'$} (LB);
		%\node at (0.5, 1) {$\ulcorner$};
		%\node at (1.5, 0.5) {$\lrcorner$};
\end{tikzpicture}
\end{center}
in which $X \to X'$ is a homotopy equivalence induces, for each $M$, a bijection between isotopy classes of $(X, \xi)$-structures on $M$ and $(X', \xi')$-structures on $M$. 

If $\xi: X \to BO(d)$ is not a fibration, then it is customary to replace it by one. For example we can replace it by $X \times_{BO(d)} PBO(d)$. Here $PY$ denotes the free path space on $Y$. Tangential structures are then defined using the replacement. For example with the suggested choice of replacement a tangential structure is a diagram as above, but where the triangle only commutes up to a specified homotopy. 

There are many examples of tangential structures (we omit $\xi$ when its clear from context):
\begin{itemize}
	\item $X = BSO(d)$, orientations;
	\item $X= BSpin(d)$, spin structures;
	\item $X = EO(d)$ or $pt$,  tangential framings (here $EO(d)$ will mean the frame bundle of the tautological bundle over $Gr_d(\R^n \times \R^\infty)$);
	\item $X = O/O(d) = \hofib( BO(d) \to BO)$, stable framings;
	\item $X = BG \times BO(d)$, via $\xi =$ projection,  $G$-principal bundles;
	\item $X = K \times BO(d)$, via $\xi =$ projection, maps to $K$;
	\item etc. 
\end{itemize}
In particular we will not assume that $X$ is connected\footnote{If $X$ does happen to be connected, then there exists a topological group $G$ with $X \simeq BG$ and $\xi$ is induced from a homomorphism of topological groups $G \to O(n)$. In that case $(X, \xi)$-structures can be interpreted as lifts of the structure group of the tangent bundle to $G$.}.

Assume that $\xi$ is a fibration. Given a point $x \in X$, we may chose a point in $EO(d)$ which maps to the image of $x$ in $BO(d)$. We get a commutative diagram:
\begin{center}
\begin{tikzpicture}
		\node (LT) at (0, 1.5) {$pt$};
		\node (LB) at (0, 0) {$EO(d)$};
		\node (RT) at (2, 1.5) {$X$};
		\node (RB) at (2, 0) {$BO(d)$};
		\draw [right hook->] (LT) -- node [left] {$\simeq$} (LB);
		\draw [->] (LT) -- node [above] {$ $} (RT);
		\draw [->>] (RT) -- node [right] {$ $} (RB);
		\draw [->] (LB) -- node [below] {$ $} (RB);
		\draw [->, dashed] (LB) -- node [below] {$ $} (RT);
		%\node at (0.5, 1) {$\ulcorner$};
		%\node at (1.5, 0.5) {$\lrcorner$};
\end{tikzpicture}
\end{center}
Since the left-hand arrow is an an acyclic cofibration and the right-hand side is a fibration, we may choose a diagonal lift as indicated in the diagram. These choices thus yield a map $x_*$ from tangential framings to $(X,\xi)$-structures.  

We will use this map to phrase the conditions of our results for field theories with arbitrary tangential structure. Up to isotopy of $(X,\xi)$-structures, the map $x_*$ is independent of the above choices and only depends on the component of $x$ in $X$. Thus if $\cZ: \Bord^{(X, \xi)}_d \to \cC$ is a field theory and $\theta$ denotes a  $d$-framing of the torus $T^{d-1}$, then the invertibility of $\cZ(T^{d-1}, x_* \theta)$ is independent of the above choices (and only depends on the component of $x \in X$).

\section{Moving-up, general tangential structures} \label{sec:moving_up_general}

The first result we need to generalize is Lemma~\ref{lem:moving_up}. We will use the notation from the proof of that lemma, and we suggest that the reader re-read the proof of Lemma~\ref{lem:moving_up} before continuing. The first part of the proof of this lemma works without change even in the presence of tangential structures. The $d$-dimensional bordism $W$ with corners still admits the necessary Morse function and can be written as a composite of handle attachements. Each handle is now equipped with an $(X, \xi)$-structure, which is inherited from the ambient manifold $W$. Thus it is enough to show that each handle, equipped with an $(X, \xi)$-structure is mapped to an invertible morphism. 

Moreover the basic philosophy for why these handles should map to invertible morphism, namely that they can be canceled, remains the same. However the specific argument must be modified slightly. The problem is that the bordism $H_{p-1}$ may not admit an $(X, \xi)$-structure extending the one on the given $(p-1)$-handle. 

We will modify the construction slightly. First in the construction of $H_{p-1}$, instead of whiskering with the bordism $D^p \times S^{q+1}$, we whisker with the bordism
\begin{equation*}
	(D^p \times S^{q+1} \setminus D^{d-1}): S^{d-2} \to S^{p-1} \times S^{q+1}.
\end{equation*}
This is the same as in Lemma~\ref{lem:moving_up}, except that we have removed a disk from this bordism. 
Let us call the result $H_{p-1}'$. Similarly in the construction of $H_p$ we instead whisker with 
\begin{equation*}
	(S^p \times D^{q+1} \setminus D^{d-1}): S^{d-2} \to S^p \times S^{q}.
\end{equation*}
We will call result $H_p'$. Now the composition yields:
\begin{equation*}
	H_p' \circ H_{p-1}' \cong (H_p \circ H_{p-1}) \setminus (D^{d-1} \times I) \cong (S^{d-1} \times I) \setminus (D^{d-1} \times I) \cong D^{d-1} \times I
\end{equation*}
that is, the identity bordism on $D^{d-1}$. 

The effect of these changes is that now the inclusion of the $(p-1)$-handle $D^p \times D^{q+2} \hookrightarrow H_p' \circ H_{p-1}'$ is a homotopy equivalence. Hence any $(X,\xi)$-structure on this $(p-1)$-handle can be extended to an $(X,\xi)$-structure on all of $H_p' \circ H_{p-1}'$, and hence on each of $H_p'$ and $H_{p-1}'$, as well as the $p$-handle. Now the proof proceeds precisely as before and yields:

\begin{lemma}\label{lem:moving_up_gentangential}
	Let $\cZ: \Bord^{(X,\xi)}_d \to \cC$ be a once-extended topological field theory such that $\cZ(Y)$ is invertible for each $(d-2)$-dimensional $(X, \xi)$-manifold $Y$, and $\cZ(\Sigma)$ is invertible for every $(d-1)$-dimensional $(X, \xi)$-bordism $\Sigma$. Then $\cZ$ is invertible. \qed
\end{lemma}

\noindent The proof of Cor.~\ref{cor:invertblepointsareinvertible} works identically in the presence of general tangential structures and yields:

\begin{corollary}\label{cor:invertblepointsareinvertible_gentangential}
	In any fully-extended topological field theory, if the values of all zero-manifolds are invertible, then the field theory is invertible.  \qed
\end{corollary}

\section{Two-dimensional theories, general tangential structure} \label{sec:base_genstr}

\subsection{Some 2-framings} \label{sec:some2frames} In Lurie's formulation and proof of the bordism hypothesis \cite{lurie-tft} tangentially framed topological field theories play a key role. It is perhaps for this reason that they have received renewed interest in recent years. The existence and enumeration of framings for low dimensional manifolds is a classical algebraic topology problem. In the context of framed topological field theories, 2-framed bordisms and the 2-framed bordism category have been carefully discussed in \cite{Douglas:2013aa, Pstragowski:2014aa} and we refer the reader to these sources. 

However there is one special class of framings which we want to highlight: 2-framings on the circle. Up to 2-framed isomorphism (which includes isotopy of 2-framing) there are an integers worth of 2-framed circles. We can see this as follows. Suppose that $\theta$ and $\theta'$ are two 2-framings of a fixed circle (i.e. framings of $\tau_{S^1} \oplus \varepsilon$). The difference between these two framings is given by a map:
\begin{equation*}
	S^1 \to GL_2(\R) \simeq O(2)
\end{equation*}
and so up to homotopy the difference lies in $\{ S^1, O(2)\} = \Z \rtimes \Z/2$. Thus on any fixed circle, up to isotopy there are precisely $ \Z \rtimes \Z/2$ framings. The $\Z/2$-factor simply measures whether the two framings induce the same orientation. 

However, since the circle admits an orientation reversing diffeomorphism, the number of abstract framed circles is divided in half. In fact there is a \emph{canonical} bijection between    2-framed isomorphism classes of 2-framed circles and the integers. In otherwords each abstract 2-framed circle has an intrinsically defined integer associated to it. This is obtained as follows. The 2-framing $\theta$ of the abstract circle $Y$ induces and orientation of $\tau_Y$ (namely the orientation which makes the isomorphism $\tau_Y \oplus \varepsilon \cong \varepsilon^2$ orientation preserving, using the standard orientation of $\varepsilon$). Since $Y$ is a 1-manifold, an orientation of $\tau_Y$ is the same as a 1-framing, which may be stabilized to obtain a new 2-framing $\overline{\theta}$. Since $\theta$ and $\overline{\theta}$ induce the same orientation, the difference (up to homotopy) of these framings is an integer:
\begin{equation*}
	[\theta] - [\overline{\theta}] \in \Z \cong [Y, SO(2)]
\end{equation*} 
which is canonically associated to the 2-framing $\theta$\footnote{Note: the identification $[Y, SO(2)] \cong \Z$ also uses the orientation of $Y$ induced by $\Theta$.}. Thus we have well-defined 2-framings $\theta_k$ of $S^1$ for each $k \in \Z$ (they all produce the same underlying orientation of $S^1$). The 2-framing $\theta_0$ corresponds to the Lie group framing of $S^1 \cong U(1)$. 

Now consider a 2-framed bordism between 1-manifolds. There is a compatibility requirement between the 2-framing of the bordism and the 2-framing of the incoming/outgoing boundary. Namely they must agree. However to compare the two 2-framings we must choose a trivialization of the normal bundle of the boundary components. There are two possibile choices and these choices differ depending on whether the boundary is incoming or outgoing: for inward boundary components the normal bundle is trivialized using an inward pointing normal vector, while for outgoing boundary components it it trivialized using an outward pointing normal vector.  

The cup and the cap bordisms (see Figure~\ref{fig:2Dbordisms}) are contractible, hence admit unique 2-framings, up to isotopy. Restricting to the boundary gives us canonical 2-framings of the circle. 
However because in one case the circle boundary component is incoming while in the other it is outgoing, they induce distinct 2-framings on the circle. These are the 2-framings $\theta_{+1}$ and $\theta_{-1}$, respectively. 

\subsection{Spherophilia} \label{sec:spherophilia}
 We do not know if the statement of our main theorem holds for all tangential structures in dimension $d=2$. However it does hold for a large class of such structures, namely those in which the 2-sphere admits such a structure. We call such structures spherophilic, meaning `sphere-loving'.

\begin{definition}\label{def:spherophilic}
	Let $\xi:X \to BO(2)$ be a tangential structure for 2-manifolds. If $X$ is connected, the we say that $(X, \xi)$ is \emph{spherophilic} if the 2-sphere $S^2$ admits an $(X, \xi)$-structure. If $X$ is disconnected, then we say it is  \emph{spherophilic} if each component is spherophilic. 
\end{definition}

\begin{example}
	Tangential 2-framings are not spherophilic. Orientations, spin structures, and stable framings are spherophilic. An important example: 3-framings are spherophilic. Here a 3-framing means we first stabilize until the bundle is 3-dimensional and then frame it. 
\end{example}

\begin{lemma}\label{lem:spherophilia}
	Let $\xi: X \to BO(2)$ be a tangential structure for 2-manifolds. The following are equivalent:
	\begin{enumerate}
		\item $(X, \xi)$ is spherophilic;
		\item For each $[x_0] \in \pi_0X$, the image of $\pi_2(X,x_0)$ in $\pi_2BO(2) \cong \Z$ contains the even integers $2 \Z$;
		\item For each $[x] \in \pi_0X$ the $(X, \xi)$-structures $x_* \theta_{+1}$ and $x_*\theta_{-1}$ are isotopic. 
	\end{enumerate}
\end{lemma}

\begin{proof}
	For simplicity we will consider the case where $X$ is connected. The case of many components only requires more bookkeeping. We let $x \in X$ be any point.  In this case we note that the cup and cap bordisms, being contractible, admit unique $(X, \xi)$-structures up to isotopy, which are induced (via $x_*$) from their unique 2-framings. Thus the boundaries of these cup and cap bordisms have $(X, \xi)$-structures given by  $x_* \theta_{+1}$ and $x_*\theta_{-1}$, respectively. If these are isotopic, then the isotopy itself may be read as an $(X, \xi)$ bordism between these circles. Composing this with the cup and cap gives an $(X, \xi)$-structure on the 2-sphere. Conversely, given an $(X, \xi)$-structure on the 2-sphere we may remove the cup and cap to obtain a cylindrical bordism between $(S^1,x_* \theta_{+1})$ and $(S^1,x_* \theta_{-1})$. This bordism in turn may be re-read as the isotopy between $x_* \theta_{+1}$ and $x_*\theta_{-1}$. This shows (1) $\Leftrightarrow$ (3).
	
	The equivalence (1) $\Leftrightarrow$ (2) follows by obstruction theory. There is a single primary obstruction to equipping $S^2$ with an $(X, \xi)$-structure which lives in 
	\begin{equation*}
		H^2(S^2; \coker(\pi_2 X \to \pi_2BO(2))).
	\end{equation*}
It may be identified with the image of the Euler class 
\begin{equation*}
	e(S^2) = 2 \in H^2(S^2; \pi_2BO(2)) \cong \Z,
\end{equation*}
and hence the primary obstruction vanishes if and only if $2 \Z$ is contained in the image of $\pi_2 X$ in $\pi_2 BO(2) \cong \Z$. 
	%\CSP{more detail here?}
\end{proof}

\begin{example} \label{Ex:restriction_spherophilia}
	Let $\xi: X \to BO(d)$ be a tangential structure for $d$-manifolds with $d>2$. Consider the tangential structure for 2-manifolds $(X_2, \xi_2)$ defined by the homotopy pull-back square:
	\begin{center}
	\begin{tikzpicture}
			\node (LT) at (0, 1.5) {$X_2$};
			\node (LB) at (0, 0) {$BO(2)$};
			\node (RT) at (2, 1.5) {$X$};
			\node (RB) at (2, 0) {$BO(d)$};
			\draw [->] (LT) -- node [left] {$\xi_2$} (LB);
			\draw [->] (LT) -- node [above] {$ $} (RT);
			\draw [->] (RT) -- node [right] {$\xi$} (RB);
			\draw [->] (LB) -- node [below] {$ $} (RB);
			\node at (0.5, 1) {$\ulcorner$};
			%\node at (1.5, 0.5) {$\lrcorner$};
	\end{tikzpicture}
	\end{center}
	Then $(X_2,\xi_2)$ is a spherophilic tangential structure. 
	
	An $(X_2,\xi_2)$-structure on a 2-manifold is an $(X,\xi)$-structure on that manifold after stablizing the tangent bundle with a rank $d-2$ trivial bundle. We can see that it is spherophilic by comparing the following portion of the long exact sequences in homotopy groups: (here $F$ is the homotopy fiber of $\xi: X \to BO(d)$):
	\begin{center}
	\begin{tikzpicture}
			\node (LT) at (0, 1.5) {$\pi_2 X_2$};
			\node (LB) at (0, 0) {$\pi_2 X$};
			\node (MT) at (3, 1.5) {$\Z \cong \pi_2 BO(2)$};
			\node (MB) at (3, 0) {$\Z / 2 \cong \pi_2 BO(d)$};

			\node (RT) at (6, 1.5) {$\pi_1 F$};
			\node (RB) at (6, 0) {$\pi_1 F$};

			\draw [->] (LT) -- node [left] {$ $} (LB);
			\draw [->>] (MT) -- node [right] {$ $} (MB);
			\draw [->] (RT) -- node [right] {$=$} (RB);
			\draw [->] (LT) -- node [above] {$ $} (MT);
			\draw [->] (MT) -- node [above] {$ $ } (RT);
			
			\draw [->] (LB) -- node [below] {$ $} (MB);
			\draw [->] (MB) -- node [below] {$ $} (RB);
			
			%\node at (0.5, 1) {$\ulcorner$};
			%\node at (1.5, 0.5) {$\lrcorner$};
	\end{tikzpicture}
	\end{center}
	a simple diagram chase shows that $2 \Z \subseteq im(\pi_2 X_2)$, and hence $(X_2,\xi_2)$ is spherophilic by Lemma~\ref{lem:spherophilia}.
\end{example}

\subsection{The base case with spherophilic tangential structures} \label{sec:spherophilic_base}

\begin{proposition}\label{pro:basecase_spherophilic}
	Let $\xi: X \to BO(2)$ be a spherophilic tangential structure for 2-manifolds, and let $\cZ: \Bord_2^{(X,\xi)} \to \cC$ be an extended topological field theory. Then $\cZ$ is invertible if and only if for each $x \in X$ we have $\cZ(S^1, x_*\theta_{1})$ is invertible.  
\end{proposition}

\begin{proof}
	For spherophilic tangential structures the proof of Prop.~\ref{pro:basecase}, which is the oriented case, carries over with very minor changes. By Corollary~\ref{cor:invertblepointsareinvertible_gentangential} it is enough to show that each point with $(X, \xi)$-structure is given an invertible value under $\cZ$ (each $(X, \xi)$ 0-manifold is a disjoint union of these). 
	
	Let $F$ be the (homotopy) fiber of $\xi:X\to BO(2)$. The set of $(X,\xi)$-structures on the point is in bijection with $\pi_0F$. If $\pi_1 X \to \pi_1BO(2)$ is surjective this coincides with $\pi_0X$, otherwise it is two copies of  $\pi_0X$. In either case the set of  $(X,\xi)$-structures on the point is exhausted by $(pt, x_* \theta_+)$ and $(pt, x_*\theta_-)$ where $\theta_{\pm}$ denotes the positive/negative 2-framing of the point and $[x] \in \pi_0X$ ranges over all components of $X$. 

For each $[x] \in \pi_0 X$ 	the objects $(pt, x_* \theta_+)$ and $(pt, x_*\theta_-)$ are dual in $\Bord_2^{X, \xi}$, and the duality is witnessed via the elbow bordisms (which serve as the unit and counit). By Lemma~\ref{lem:invertible_unit} the points will take invertible values precisely if these elbow bordisms take invertible values under $\cZ$. 

As an abstract manifold each elbow is just a contractible disk, and so admits a unique $(X, \xi)$-structure for each component $[x] \in \pi_0X$. However as a bordism we must \emph{parametrize} the boundary and this means that as bordisms there may be multiple `left-elbows' and multiple `right-elbows' with the same source and target objects. In fact it is easy to see via obstruction theory  that set of `left-elbows' bordisms (respectively `right-elbows' bordisms) is a torsor over $\pi_1F$. 
Fortunately each of these elbows differs by composition with an invertible 1-morphism in $\Bord_2^{(X,\xi)}$ and so for questions of invertibility it is sufficient to show that any single pair of `left-elbow' and `right-elbow' takes an invertible value under $\cZ$, for then they all take invertible values. 
	
For this we will consider a pair of elbows which are adjoint to each other. The unit and counit of the adjunction are witnessed by a saddle and cup bordism. Again by Lemma~\ref{lem:invertible_unit} it is enough to show that these saddle and cup bordisms take invertible values under $\cZ$. 

Fix a left-elbow bordism. The inclusion of the left-elbow into the cup is a homotopy equivalence. Hence an $(X,\xi)$-structure on the left-elbow bordism extends to unique $(X, \xi)$-structure on the cup, which then restricts to an $(X,\xi)$-structure on the other half its circle boundary, a right-elbow. This determines an adjoint pair of left- and right-elbrows with $(X, \xi)$-structure. 

By assumption we know that the boundary of the cup bordism, $(S^1, x_* \theta_{1})$, takes an invertible value under $\cZ$. Thus the annulus, which witnesses the duality between $(S^1, x_* \theta_{1})$ and $(S^1, x_* \theta_{-1})$, is assigned an invertible value under $\cZ$. The annulus is a composite of the cup bordism and a particular pair of pants bordism. All the 1-morphisms which are sources and targets of these pants and cup bordisms are assigned invertible values under $\cZ$ and so by Lemma~\ref{lem:symmetric_mon_invert_mor}, it follows that both this pair of pants and the cup bordism are assigned invertible values under $\cZ$. The dual argument shows that the cap bordism is also assigned an invertible value. 

So all that remains is to show that the saddle bordism takes an invertible value. This is where we need to use the face that $(X, \xi)$ is a spherophilic tangential structure. For a general tangential structure the left and right adjoint of a fixed left-elbow bordism may be distinct. But for a spherohilic tangential structure they are necessarily the same (and the composite of these elbows into a circle yields the $(X,\xi)$-stricture $x_* \theta_1 \simeq x_* \theta_{-1}$ on the circle). Thus all of the saddle bordisms used in the caclulations depicted in Figures~\ref{fig:decomposition_coposedSaddles1} and \ref{fig:decomposition_coposedSaddles2} admit (unique) $(X, \xi)$-structures making them composable. Moreover the key identity, depicted in Figure~\ref{fig:Cylinder-Sphere-eqn} also holds, and so these computations remain valid when $(X, \xi)$ is spherophilic. 
\end{proof}

It would be interesting to know if the above result is sharp. As of this version of this paper we have been unable to decide either way and so offer a conjecture:

\begin{conjecture}
	There exists some symmetric monoidal bicategory $\cC$ and a 2-framed 2-dimensional topological field theory
	\begin{equation*}
		\cZ: \Bord_2^\text{fr} \to \cC
	\end{equation*}
	such that $\cZ(S^1, \theta_1)$ is invertible, but such that the $\cZ$ is \emph{not invertible}. 
\end{conjecture}

\section{Dimensional reduction, revisited} \label{sec:dimensional_redux}

One of the key techniques which we used to prove our main theorem in the oriented case was the technique of dimensional reduction. There are many versions of dimensional reduction, and while we only needed a simple form in the oriented case, we will need more complicated versions in the case of general tangential structures. 

\subsection{Basic dimensional reduction} \label{sec:dimredux_basic}

The simplest form of dimensional reduction which works for arbitrary tangential structures happens by taking the product with a tangentially framed manifold. Let $BO(k) \to BO(d)$ be the map induced by adding $d-k$ trivial line bundles, and let $X_{k}$ denote the pullback:
\begin{center}
\begin{tikzpicture}
		\node (LT) at (0, 1.5) {$X_{k}$};
		\node (LB) at (0, 0) {$BO(k)$};
		\node (RT) at (2, 1.5) {$X$};
		\node (RB) at (2, 0) {$BO(d)$};
		\draw [->] (LT) -- node [left] {$\xi_{k}$} (LB);
		\draw [->] (LT) -- node [above] {$ $} (RT);
		\draw [->] (RT) -- node [right] {$\xi$} (RB);
		\draw [->] (LB) -- node [below] {$ $} (RB);
		\node at (0.5, 1) {$\ulcorner$};
		%\node at (1.5, 0.5) {$\lrcorner$};
\end{tikzpicture}
\end{center}
It $(M, \theta)$ is a tangentially framed $(d-k)$-manifold and $(Y, \psi)$ is a $k$-manifold with an $(X_{k}, \xi_{k})$-structure, then the product $(Y \times M, \psi \times \theta)$ is naturally a $d$-manifold with an $(X, \xi)$-structure. This gives rise to a functor:
\begin{equation*}
	(-) \times (M, \theta): \Bord^{(X_{k}, \xi_{k})}_{k} \to \Bord_d^{(X, \xi)}
\end{equation*}
which can be used to preform dimensional reduction. 
Dually, the same construction also gives rise to a functor
\begin{equation*}
	(Y, \psi) \times (-): \Bord_{d-k}^\text{fr} \to \Bord_d^{(X, \xi)}
\end{equation*}
from the tangentially framed $(d-k)$-dimensional bordism higher category. 

These dual forms of dimensional reduction are actually part of a more general context. Consider the following situation. Suppose that $\xi_a: X_a \to BO(k)$ is a tangential structure for $k$-manifolds and $\xi_b:X_b \to BO(d-k)$ is a tangential structure for $(d-k)$-manifolds. Suppose further that we have a commutative diagram:
\begin{center}
\begin{tikzpicture}
		\node (LT) at (0, 1.5) {$X_a \times X_b$};
		\node (LB) at (0, 0) {$BO(k) \times BO(d-k)$};
		\node (RT) at (3, 1.5) {$X$};
		\node (RB) at (3, 0) {$BO(d)$};
		\draw [->] (LT) -- node [left] {$\xi_a \times \xi_b$} (LB);
		\draw [->] (LT) -- node [above] {$f$} (RT);
		\draw [->] (RT) -- node [right] {$\xi$} (RB);
		\draw [->] (LB) -- node [below] {$ $} (RB);
		%\node at (0.5, 1) {$\ulcorner$};
		%\node at (1.5, 0.5) {$\lrcorner$};
\end{tikzpicture}
\end{center}
Then we have a pairing. If $(Y, \psi)$ is a $y$-manifold with $(X_a, \xi_a)$-structure and $(M, \theta)$ is a $(d-k)$-manifold with $(X_b, \xi_b)$-structure, then $(Y \times M, f_*(\psi \times \theta))$ is a $d$-manifold with $(X, \xi)$-structure. This gives rise to functors:
\begin{align*}
	(-) \times (M, \theta): \Bord^{(X_a, \xi_a)}_{k} &\to \Bord_d^{(X, \xi)} \\
	(Y, \psi) \times (-): \Bord_{d-k}^{(X_b, \xi_b)} &\to \Bord_d^{(X, \xi)}.
\end{align*}
This was exactly the sort of dimensional reduction used in the oriented case where $(X_a,\xi_a)$ and $(X_b, \xi_b)$ both corresponded to the structure of orientations. 

\subsection{Total dimensional reduction} \label{sec:dimredux_total} There is another kind of dimensional reduction which we will need to use in order to prove our main theorem in the presence of general tangential structures. The basic dimensional reduction, described above, splits the problem of constructing an $(X, \xi)$-structure on $Y \times M$ into finding two different and separate tangential structures, one on $Y$ and one on $M$. However we don't need to separate these. The bare (unstructured) $(d-k)$-manifold $M$ defines a new kind of tangential structure $(X_M, \xi_M)$ for $k$-manifolds. An  $(X_M, \xi_M)$-structure on $Y$ is exactly an $(X, \xi)$-structure on $Y \times M$. 

To describe this new structure first fix a $(d-k)$-manifold $M$, and consider the induced fiber sequence:
\begin{equation*}
	F_M \to \Map(M, X) \to \Map(M, BO(d))
\end{equation*}
where the fiber is taken over the map $\tau_M \oplus \varepsilon^{k}$. The fiber $F_M$ is the `space of $(X, \xi)$-structures on $M$'. The components of $F_M$ are in natural bijection with isotopy classes of $(X, \xi)$-structures on $M$.

There is a map $BO(k) \times \Map(M, BO(d-k)) \to \Map(M, BO(d))$ which is induced by applying the direct sum map
\begin{equation*}
	BO(k) \times BO(d-k) \to BO(d)
\end{equation*}
pointwise in $M$. From this we can construct $\xi_M:X_M \to BO(k)$ via the following pullback square:
\begin{center}
\begin{tikzpicture}
		\node (LT) at (0, 1.5) {$X_M$};
		\node (LB) at (0, 0) {$BO(k) \times \{\tau_M\}$};
		\node (MB) at (5, 0) {$BO(k) \times \Map(M, BO(d-k))$};
		\node (RT) at (10, 1.5) {$ \Map(M, X)$};
		\node (RB) at (10, 0) {$ \Map(M, BO(d))$};
		\draw [->] (LT) -- node [left] {$\xi_M$} (LB);
		\draw [->] (LT) -- node [above] {$ $} (RT);
		\draw [->] (RT) -- node [right] {$ $} (RB);
		\draw [->] (LB) -- node [below] {$ $} (MB);
		\draw [->] (MB) -- node [below] {$ $} (RB);
		\node at (0.5, 1) {$\ulcorner$};
		%\node at (1.5, 0.5) {$\lrcorner$};
\end{tikzpicture}
\end{center}
A lift $\theta: T \to X_M$ over $\tau_Y: Y \to BO(k)$ is the same as a lift of $\tau_Y \oplus \tau_M: Y \times M \to BO(d)$ to $X$. 
This gives rise to a new dimensional reduction functor: 
\begin{equation*}
	(-) \times M: \Bord_k^{(X_M, \xi_M)} \to \Bord_d^{(X, \xi)}.
\end{equation*}

Note that there is a surjective map $\pi_0 F_M \to \pi_0 X_M$.  It is either a bijection or a two-to-one mapping, and hence each $(X, \xi)$-structure on $M$ singles out a component of $X_M$; every component is realized this way.

\subsection{A variation on total dimensional reduction along a circle.} \label{sec:dimredux_nullholonomic}

 Total dimensional reduction, described above constructs a new tangential structure $(X_M, \xi_M)$ from an initial tangential structure $(X, \xi)$ and a manifold $M$. However the new tangential structure can have many components if $M$ admits many $(X, \xi)$-structures. 
 
 For example consider the case $M=S^1$ of total dimensional reduction along a circle. Let us compute the number of components of $X_{S^1}$. Let $F$ be the homotopy fiber of the map $\xi:X \to BO(d)$. We have a long exact sequence:
 \begin{equation*}
 	\cdots \to \pi_2 BO(d) \to \pi_1 F \to \pi_1 X \to \pi_1 BO(d) \to \pi_0 F \twoheadrightarrow \pi_0 X
 \end{equation*} 
 The tangent bundle of $S^1$ is trivializable, and hence the stabilized map $\tau_{S^1} \oplus\varepsilon^{d-1}: S^1 \to BO(d)$ is null-homotopic. If we choose a null-homotopy of this map, then this gives us an identification of $F_{S^1}$ with $LF = \Map(S^1, F)$, the free loop-space. It then follows, from the long exact homotopy sequence for the fibration $F_{S^1} \to X_{S^1} \to BO(d-1)$, that we have a bijection:
 \begin{equation*}
 	\pi_0 X_{S^1} \cong \pi_0 X \times \pi_1 F.
 \end{equation*}
Thus the number of components of $X_{S^1}$ grows multiplicatively by a factor of size $\pi_1F$. 
We would like to describe a modification of the total dimensional reduction which will cut this number down, but still allow more flexibility than the basic dimensional reduction we have already seen. 

The identification $\pi_0 X_{S^1} \cong \pi_0 X \times \pi_1 F$ is not canonical, but depends on our choice of a null-homotopy of $\tau_{S^1} \oplus\varepsilon^{d-1}$. Two such null-homotopies differ by an element in $\pi_2 BO(d)$, and this may change the above identification by translation by the image of  $\pi_2 BO(d)$ in $\pi_1 F$. Thus we get an invariant of $(X, \xi)$-structures on the circle taking values in 
\begin{equation*}
	\pi_0 X \times (\pi_1F /\pi_2 BO(d))  \subseteq \pi_0 X \times \pi_1 X
\end{equation*}
This invariant may be read off from a given $(X, \xi)$-structure (which recall is a certain map $\theta: S^1 \to X$) by looking at the action on $\pi_0$ and $\pi_1$ induced by $\theta$. The circle is connected and so $\theta$ distinguishes a component of $X$, and the element in $\pi_1X$ is the image of the generator of $\pi_1 S^1 \cong \Z$ (since $\theta$ lifts the stable tangent bundle of $S^1$, this automatically lands in the above subgroup of $\pi_1 X$, the kernel of the map to $\pi_1 BO(d)$). 

If we want to consider $(X, \xi)$-structures on any abstract circle, then we must further quotient by the effect of $\pi_0 Diff(S^1) \cong \Z/2$ (the non-trivial component corresponds to orientation reversing diffeomorphisms). This acts on the $\pi_1 X$ factor by sending an element to its inverse. Hence the unordered pair $\{ g, g^{-1}\}$ is still a well defined invariant of an $(X, \xi)$-structure on an abstract circle. We will call this the \emph{holonomy} of the $(X, \xi)$-structure. 

In the total dimensional reduction we consider $(X_{S^1}, \xi_{S^1})$-structures on manifolds $Y$ which are the same as $(X, \xi)$-structures on $Y \times S^1$. For each point $y \in Y$ and each framing of $T_yY$ we get an induced $(X, \xi)$-structure on $\{ y \} \times S^1$, and we can read off the holonomy of this factor. (If the holonomy is consider as an unordered pair of elements $\{g, g^{-1}\} \subseteq \pi_1X$, then this doesn't depend on the choice of framing of $T_y Y$). 

We will now describe a new tangential structure $(\overline{X}_{S^1}, \overline{\xi}_{S^1})$ on $(d-1)$-manifolds $Y$, where such a structure is an $(X, \xi)$-structure on $Y \times S^1$ such that around each $\{y\} \times S^1$, the induced $(X, \xi)$-structure has null-holonomy. 

Let $\Map_0(S^1, X)$ denote the union of the components of $\Map(S^1, X)$ such that the induced map $\pi_1 S^1 \to \pi_1 X$ is the zero-homomorphism. Then, mimicking the construction for total dimensional reduction, we form $(\overline{X}_{S^1}, \overline{\xi}_{S^1})$ as the homotopy pull-back:
\begin{center}
\begin{tikzpicture}
		\node (LT) at (0, 1.5) {$\overline{X}_{S^1}$};
		\node (LB) at (0, 0) {$BO(d-1) \times \{\tau_{S^1}\}$};
		\node (MB) at (5, 0) {$BO(d-1) \times \Map(S^1, BO(1))$};
		\node (RT) at (10, 1.5) {$ \Map_0(S^1, X)$};
		\node (RB) at (10, 0) {$ \Map(S^1, BO(d))$};
		\draw [->] (LT) -- node [left] {$\overline{\xi}_{S^1}$} (LB);
		\draw [->] (LT) -- node [above] {$ $} (RT);
		\draw [->] (RT) -- node [right] {$ $} (RB);
		\draw [->] (LB) -- node [below] {$ $} (MB);
		\draw [->] (MB) -- node [below] {$ $} (RB);
		\node at (0.5, 1) {$\ulcorner$};
		%\node at (1.5, 0.5) {$\lrcorner$};
\end{tikzpicture}
\end{center}
We get an induced functor which allows us to preform \emph{null-holonomic dimensional reduction along $S^1$}:
\begin{equation*}
	(-) \times S^1: \Bord_{d-1}^{(\overline{X}_{S^1}, \overline{\xi}_{S^1})} \to \Bord_d^{(X, \xi)}.
\end{equation*}
  
 \begin{example} \label{ex:nullholonomicdimredux}
 	Let $\theta$ be a $k$-framing of $S^1$ and let $(Y, \psi)$ be $d-k$-dimensional manifold with an $(X, \xi)$-structure. Then we have an induced $(X, \xi)$-structure on $S^1 \times Y$ via:
	\begin{equation*}
		\tau_{S^1} \oplus \varepsilon^{(k-1) \oplus } \oplus  \tau_Y \stackrel{\theta}{\cong} \varepsilon^{k \oplus } \oplus \tau_Y
	\end{equation*}
and then pointwise application of $\psi$. Since $\psi$ is applied pointwise, the induced $(X, \xi)$-structures on $S^1 \times \{y\}$ have null-holonomy. Hence this defines an $(\overline{X}_{S^1}, \overline{\xi}_{S^1})$-structure on $Y$. 	
	 \end{example}
 
\begin{remark} \label{rmk:nullholonomicdimredux}
 By construction,  $\overline{X}_{S^1}$ consists of a collection of certain components of $X_{S^1}$. 
	The number of components of $\overline{X}_{S^1}$ can also be computed, as we did above for $X_{S^1}$. We see that we have a non-canonical bijection
	\begin{equation*}
		\pi_0 \overline{X}_{S^1} \cong \pi_0X \times im(\pi_2 BO(d)) \subseteq \pi_0 X \times \pi_1 F.
	\end{equation*}
	Again this bijection depends on the choice of a null-homotopy of the stable tangent bundle of $S^1$. These are given by the null-holonomic $(X, \xi)$-structures on $S^1$, and we see, in particular, that up to isotopy these are exhausted by the $(X, \xi)$-structures $(S^1, x_* \theta)$ where $\theta$ is a $d$-framing of $S^1$. We also remark that since $\overline{X}_{S^1}$ consists of a collection of certain components of $X_{S^1}$, two $(\overline{X}_{S^1}, \overline{\xi}_{S^1})$-structures are isotopic as $(\overline{X}_{S^1}, \overline{\xi}_{S^1})$-structures if and only if they are isotopic as $(X_{S^1}, \xi_{S^1})$-structures.  
\end{remark}

\subsection{Dimensional reduction to spherophilic structures}

Finally we will consider total dimensional reduction along a $(d-2)$-dimensional manifold $M$. That is we want to consider the functor:
\begin{equation*}
 (-) \times M:	\Bord_2^{(X_M, \xi_M)} \to \Bord_d^{(X, \xi)}.
\end{equation*}
In view of Prop.~\ref{pro:basecase_spherophilic} we would like to know when the new tangential structure $(X_M, \xi_M)$ is spherophilic? Of course a complete answer to this might depend on the particular tangential structure $(X, \xi)$ that we started with. However for certain choices of $M$ it turns out that $(X_M, \xi_M)$ will be spherophilic irregardless of the initial tangential structure. 

\begin{lemma}\label{lem:redux_to_spherophilia}
	Let $\xi:X \to BO(d)$ be any tangential structure for $d$-manifolds. Let $M$ be a $(d-2)$-manifold. Let $\xi_M: X_M \to BO(2)$ be the corresponding dimensionally reduced tangential structure for 2-manifolds.  If the top Stiefel-Whitney class $w_{d-2}(M) = 0$ vanishes (equivalently $\chi(M)$ is even), then $(X_M, \xi_M)$ is spherophilic. 
\end{lemma}

\begin{proof}
	If $M$ does not admit any $(X, \xi)$-structures, then $X_M$ is empty and therefore vacuously spherophilic. So we will suppose that $\psi$ is an $(X, \xi)$-structure for $M$. This structure corresponds to a component of $F_M$ and hence to a component of $X_M$ via the projection $\pi_0 F_M \to \pi_0 X_M$. As we have seen each component of $X_M$ gives a map from 2-framings to $(X_M, \xi_M)$-structures (and hence to $(X, \xi)$-structures on the product of the manifold with $M$). We will denote this map by $\psi_*$. Thus for example we have $(X_M, \xi_M)$-manifolds: $(S^1, \psi_* \theta_{1})$ and $(S^1, \psi_* \theta_{-1})$, where $\theta_k$ denotes the $k^\text{th}$ 2-framing of the circle (see Section~\ref{sec:some2frames}). By Lemma~\ref{lem:spherophilia} it is sufficient to show for each $\psi$ ($(X, \xi)$-structure on $M$) that $ \psi_* \theta_{1}$ and $ \psi_* \theta_{-1}$ are isotopic as $(X_M, \xi_M)$-structures. 

Let us consider the $(X_M, \xi_M)$ structures further, viewing them as $(X, \xi)$-structures on $S^1 \times M $. Each structure is obtaining in two steps. First the 2-framings of the circle give us two identifications:
\begin{equation*}
	\overline{\theta}_{+ 1},\overline{\theta}_{-1},  :  \tau_{S^1} \oplus \varepsilon \oplus \tau_M  \stackrel{\cong}{\to} \varepsilon^{\oplus 2} \oplus\tau_M.
\end{equation*}
Here we use $\overline{\theta}_{\pm1}$ to distinguish these induced maps from the 2-framings themselves. After this identification we use the $(X, \xi)$-structure $\psi$ from $M$ (pointwise in the $S^1$-coordinate) to get an $(X, \xi)$-structure on $M \times S^1$. Thus we would be done if it happens that the first two identifications are isotopic. 

The two 2-framings $\theta_{+1}$ and $\theta_{-1}$ become isotopic after stabilizing to 3-framings, that is after adding a trivial line bundle. Thus for example if the tangent bundle of $M$ splits off a  trivial line bundle, $\tau_M \cong \varepsilon \oplus E$, then the identifications $\overline{\theta}_{+ 1}$ and $\overline{\theta}_{-1}$ are isotopic and we would be done. 

The obstruction to $\tau_M$ decomposing in this way is well-known to be the Euler class of the manifold $M$, and hence such a splitting occurs only if the Euler characteristic of $M$ vanishes. This covers, for example, the case that $d$ is odd. 

However we can do better. In this argument it is not strictly necessary that $\tau_M$ splits a trivial line bundle; this only needs to happen stably. That is, it is sufficient to know that $\varepsilon \oplus \tau_M  \cong \varepsilon^{\oplus 2} \oplus E$ for some rank $(d-3)$-bundle $E$ on $M$. 

For $(d-2)$-manifolds there is a single obstruction to the existence of such a splitting which may be identified with the mod 2 reduction of the Euler characteristic, a.k.a. the $(d-2)$-dimensional Stiefel-Whitney class $w_{d-2}$. 
For example we can obtain this identification by using obstruction theory and comparing via the map of homotopy fiber sequences:
\begin{center}
\begin{tikzpicture}
		\node (LT) at (0, 1.5) {$S^{d-3}$};
		\node (LB) at (0, 0) {$V_2(\R^{d-1})$};
		\node (MT) at (2.5, 1.5) {$BO(d-3)$};
		\node (MB) at (2.5, 0) {$BO(d-3)$};
		\node (RT) at (5, 1.5) {$BO(d-2)$};
		\node (RB) at (5, 0) {$BO(d-1)$};
		\draw [->] (LT) -- node [left] {$ $} (LB);
		\draw [->] (LT) -- node [above] {$ $} (MT);
		\draw [->] (LB) -- node [below] {$ $} (MB);
		\draw [->] (MT) -- node [left] {$ $} (MB);
		\draw [->] (MT) -- node [above] {$ $} (RT);
		\draw [->] (MB) -- node [below] {$ $} (RB);
		\draw [->] (RT) -- node [right] {$ $} (RB);
		%\node at (0.5, 1) {$\ulcorner$};
		%\node at (1.5, 0.5) {$\lrcorner$};
\end{tikzpicture}
\end{center}
The important map, in the non-trivial case that $d$ is even, is the surjection $\Z \cong \pi_{d-3} S^{d-3} \to \pi_{d-3} V_2(\R^{d-1}) \cong \Z/2$. 
Thus if the characteristic number $w_{d-1}(M) = 0$, the tangential structure $(X_M, \xi_M)$ is spherophilic for any tangential structure $(X, \xi)$. 
\end{proof}

\begin{corollary}\label{cor:nullholonomic_dim_reduc_spherophilia}
	Suppose that $\xi:X \to BO(3)$ is a tangential structure for 3-manifolds and let $(\overline{X}_{S^1}, \overline{\xi}_{S^1})$ denote the corresponding null-holonomic dimensionally reduced structure for 2-manifolds described in Section~\ref{sec:dimredux_nullholonomic}. Then $(\overline{X}_{S^1}, \overline{\xi}_{S^1})$ is a spherophilic tangential structure. 
\end{corollary}

\begin{proof}
	Since the Euler characteristic of $S^1$ is zero, the structure for total dimensional reduction $(X_{S^1}, \xi_{S^1})$ is spherophilic. In particular for each $y \in \pi_0 X_{S^1}$ we have $y_* \theta_{+1} \simeq y_* \theta_{-1}$ are isotopic $(X_{S^1}, \xi_{S^1})$-structures on $S^1$. Since $\overline{X}_{S^1}$ consists of a collection of components of $X_{S^1}$, it follows that $y_* \theta_{+1} \simeq y_* \theta_{-1}$ are isotopic as $(\overline{X}_{S^1}, \overline{\xi}_{S^1})$-structures whenever $y$ belongs to the componets making up $\overline{X}_{S^1}$, see Remark~\ref{rmk:nullholonomicdimredux}. Thus $(\overline{X}_{S^1}, \overline{\xi}_{S^1})$ is a spherophilic.
\end{proof}

\section{The main theorem, general tangential structure} \label{sec:induct_general}

We are now set to prove our main theorem in the presence of general tangential structures.

\begin{theorem}\label{thm:MainThm_general}
	Fix $d \geq 3$ and any tangential structure $\xi: X \to BO(d)$. Consider a once-extended topological field theory
	\begin{equation*}
		\cZ: \Bord_d^{(X, \xi)} \to \cC.
	\end{equation*}
	Then $\cZ$ is invertible if and only if for each component $[x] \in \pi_0 X$ the value of $\cZ(T^{d-1}, x_* \theta_{+1} \times \theta_\text{Lie})$ is invertible. 
\end{theorem}

\begin{proof}
	In very broad strokes the proof here is the same as for Theorem~\ref{thm:main_thm_SO} in the oriented case, however there are a number of small alterations and side arguments that must be made when we are dealing with general tangential structures. Recall that in the proof of the oriented case we had the following steps:
	\begin{enumerate}
		\item We used dimensional reduction along $S^1$ and induction to show that $\cZ(M \times S^1)$ is invertible for any $(d-2)$-manifold $M$;
		\item We used dimensional reduction along $M$ and the base case to show that $\cZ(M)$ is invertible for any $(d-2)$-manifold $M$;
		\item We showed that the invertibility of closed $(d-1)$-manifolds under $\cZ$ was invariant under surgery; and
		\item We used a variety of tricks to extend this to all $(d-1)$-bordisms and hence proved the theorem using Lemma~\ref{lem:moving_up}. 
	\end{enumerate}
The first and most significant difficulty with duplicating this argument in the presence of general tangential structures is that we have only established the base case ($d=2$) for spherophilic tangential structures and not for all tangential structures (See Section~\ref{sec:spherophilic_base}). This complicates both the  argument in step (2) and the
induction in step (1) (particularly in the next-to-lowest $d=3$ case). 

This also necessitates using the more complicated dimensional reductions described in Sections~\ref{sec:dimredux_total} and \ref{sec:dimredux_nullholonomic}, rather than the basic dimensional reduction described in Section~\ref{sec:dimredux_basic}. For example in step (2) we want to conclude that $\cZ(M, \theta)$ is invertible for any $(X, \xi)$-structure $\theta$ on $M$. The basic dimensional reduction along $(M, \theta)$ yields a tangetially framed 2-dimensional field theory:
\begin{equation*}
	Z_{(M, \theta)}: \Bord_2^\text{fr} \to \cC.
\end{equation*}
However 2-framings are not a spherophilic tangential structure, and hence we can't appeal to Proposition~\ref{pro:basecase_spherophilic}. Using total dimensional reduction instead allows us to correct this in some cases (namely when $w_{d-2}(M) = 0$, see Lemma~\ref{lem:redux_to_spherophilia}).  

However total dimensional reduction also has its pitfalls. For example in step~1, we would like to dimensionally reduce along the circle and appeal to induction to show that this new dimensionally reduced theory is invertible. Using total dimensional reduction along the circle at first seems promising. For example in the lowest case $d=3$, since the Euler characteristic of the circle $\chi(S^1) = 0$ vanishes, by Lemma~\ref{lem:redux_to_spherophilia} total dimensional reduction along the circle yields a spherophilic tangential structure irregardless of $(X, \xi)$. However there is another, different problem in trying to apply induction. 

Let $\cZ_{S^1}$ temporarily denote the effect of doing total dimensional reduction along $S^1$ to the theory $\cZ$. By assumption we know that the value $\cZ(T^{d-1}, x_* \theta_{+1} \times \theta_\text{Lie})$ is invertible for each component $[x] \in \pi_0 X$. However to apply our induction hypothesis to we would need to know the invertibility of $\pi_0 X_{S^1}$-many morphisms. As we saw in Section~\ref{sec:dimredux_nullholonomic}, $\pi_0 X_{S^1} \cong \pi_0 X \times \pi_1 F$, where $F$ is the homotopy fiber of $\xi: X \to BO(d)$. Depending on $X$ this can yield more conditions than we have assumptions, and so we cannot apply induction in this way (at least not for general $X$). 

The solution for step~1 is to use the null-holonomic dimensional reduction which was described in Section~\ref{sec:dimredux_nullholonomic}. That is we precompose $\cZ$ with the functor
\begin{equation*}
	(-) \times S^1: \Bord_{d-1}^{(\overline{X}_{S^1}, \overline{\xi}_{S^1})} \to \Bord_d^{(X, \xi)}
\end{equation*}
to obtain a new field theory, which we now denote $\cZ_{{S^1}}$, for manifolds with $(\overline{X}_{S^1}, \overline{\xi}_{S^1})$-structures. As we saw, the components of $\overline{X}_{S^1}$ are in bijection with \emph{null-holonomic} $(X, \xi)$-structures on $S^1$ and these are exhausted by $(S^1, x_* \psi)$ where $\psi$ is a $d$-framing of $S^1$. 

Let $y = (x, \psi)$ be a pair consisting of a point $x \in X$ and a $d$-framing of $S^1$. Let $(T^{d-2}, \theta_1 \times \theta_\text{Lie})$ denote the $(d-2)$-torus with $(d-1)$-framing which is the positive (bounding) 2-framing $\theta_1$ on the first circle and the Lie group framing on the remaining factors (see Section~\ref{sec:some2frames}). For $d > 3$, to apply our induction hypothesis it suffices to know that $\cZ_{S^1}( T^{d-2}, y_* \theta_1 \times \theta_\text{Lie})$ is invertible for all $y$. Computing we have
\begin{equation*}
	\cZ_{S^1}( T^{d-2}, y_* \theta_1 \times \theta_\text{Lie}) = \cZ(T^{d-2} \times S^1, x_*(\theta_1 \times \theta_\text{Lie} \circ \psi))
\end{equation*}
where $\theta_1 \times \theta_\text{Lie} \circ \psi$ denotes the $d$-framing on $T^{d-2} \times S^1$ obtained as a composite:
\begin{equation*}
	\varepsilon \oplus \tau_{T^{d-2}} \oplus \tau_{S^1} \stackrel{\theta_1 \times \theta_\text{Lie}}{\cong} \varepsilon^{\oplus d-1} \oplus \tau_{S^1} \stackrel{\psi}{\cong} \varepsilon^{\oplus d}.
\end{equation*}
We know by assumption that $\cZ(T^{d-1}, x_* \theta_1 \times \theta_\text{Lie})$ is invertible. The $d$-framings $\theta_1 \times \theta_\text{Lie} \circ \psi$  and $\theta_1 \times \theta_\text{Lie}$ on $T^{d-1}$ may not be isotopic, but nevertheless the resulting $d$-framed tori are framed diffeomorphic\footnote{In fact the framed diffeomorphism is supported on a 2-dimensional stably framed torus and so it suffices to consider that case. There are precisely four stable framings on $T^2$, and under the action of the diffeomorphisms of $T^2$ three of these are permuted. The Lie group framing is the single fixed point. In the $(d-1)$-tori case, the relevant framings are products which differ only on at most a 2-torus, but since there is always a $\theta_1$-factor, these framings on 2-tori are give diffeomorphic framed tori.}, and hence (provided $d >3$) the conditions of our induction hypothesis are satisfied. When $d=3$ we need to check the additional condition that the structure $(\overline{X}_{S^1}, \overline{\xi}_{S^1})$ is spherophilic, but this is always the case by Corollary~\ref{cor:nullholonomic_dim_reduc_spherophilia}.

Thus the dimensionally reduced theory $\cZ_{S^1}$ is an invertible theory. This implies that for any $(d-2)$-manifold $M$ and $(\overline{X}_{S^1}, \overline{\xi}_{S^1})$-structure $\psi$ on $M$, the value $\cZ_{S^1}(M, \psi)$ is invertible. An example of such a $\psi$ was given in Example~\ref{ex:nullholonomicdimredux}: if $\theta$ is any $(X, \xi)$-structure on $M$ and $\theta_k$ is any 2-framing of $S^1$, then the induced $(X, \xi)$-structure on $M \times S^1$, which we will denote $\theta \times \theta_k$ constitutes such an $(\overline{X}_{S^1}, \overline{\xi}_{S^1})$-structure $\psi$ on $M$. It follows that
\begin{equation*}
	\cZ(M \times S^1, \theta \times \theta_k)
\end{equation*}
is invertible for any $(X, \xi)$-structure $\theta$ on $M$ and 2-framing $\theta_k$ on $S^1$. 

Next, we proceed to step~2 and consider total dimensional reduction along $(d-2)$-manifolds $M$. That is we precompose $\cZ$ with the functor 
\begin{equation*}
	M \times (-):\Bord_2^{(X_M, \xi_M)} \to \Bord_d^{(X, \xi)}
\end{equation*}
to obtain a new 2-dimensional theory $\cZ_M$ for $(X_M, \xi_M)$-structures. We will attempt to show that this theory is invertible by appealing to Proposition~\ref{pro:basecase_spherophilic}. To apply this proposition we need to show two things, first that for each component $y \in \pi_0 X_M$ the value $\cZ_M(S^1, y_* \theta_1)$ is invertible, and second that $(X_M, \xi_M)$ is a spherophilic tangential structure. 

Let us consider the first condition first. As explained in Section~\ref{sec:dimredux_total} the components of $X_M$ receive a surjective map from the set of $(X, \xi)$-structures on $M$. If $\theta$ is such a structure (mapping to $[y] \in \pi_0 X_M$) and $\theta_k$ is a 2-framing of $S^1$ then we obtain an induced $(X_M, \xi_M)$-structure $y_* \theta_k$ on $S^1$. This corresponds precisely to the $(X, \xi)$-structure on $M \times S^1$ which we denoted by $\theta \times \theta_k$ above. In particular we have already established that 
\begin{equation*}
	\cZ_M(S^1, y_* \theta_k) = \cZ(M \times S^1, \theta \times \theta_k)
\end{equation*} 
is invertible. 

The second condition is more problematic, but Lemma~\ref{lem:redux_to_spherophilia} ensures that $(X_M, \xi_M)$ is spherophilic provided that the top Stiefel-Whitney class vanishes, $w_{d-2}(M) = 0$. In that case Proposition~\ref{pro:basecase_spherophilic} tells us that the dimensionally reduced theory $\cZ_M$ is invertible. In particular we have shown that if $M$ is any $(d-2)$-manifold such that $w_{d-2}(M) = 0$ (i.e. each component of $M$ has even Euler characteristic), and $\theta$ is any $(X, \xi)$-structure on $M$, then $\cZ(M, \theta)$ is invertible. For example when $d$ is odd this first condition is always satisfied. When $d$ is even, this is not yet as comprehensive a result as in the oriented case, but it is a start. In particular this shows that for all $p + q = d-2$, and all $(X, \xi)$-structures $\theta$ on $S^p \times S^q$, the value $\cZ(S^p \times S^q, \theta)$ is invertible. 

Our next goal will be to prove an analog of Lemma~\ref{lem:surgery_invertibility}. Since we only established a partial version of Step~2, we proceed with a different argument than in the oriented case. We will use the basic dimensional reduction described in Section~\ref{sec:dimredux_basic} to show:

\begin{lemma}\label{lem:knots_invertible}
	Let $p + q = d-2$ be non-negative integers and fix an $(X, \xi)$-structure $\psi$ on $S^p$. There is a unique $(q+2)$ framing $\theta$ on the bordism $D^{q+1}$, viewed as a bordism from $S^q$ to $\emptyset$. This induces (by the basic dimensional reduction map) an $(X, \xi)$-structure $\psi \times \theta$ on the bordism
	\begin{equation*}
		S^p \times D^{q+1}: S^p \times S^q \to \emptyset.
	\end{equation*} 
Let $\cZ$ be a field theory satisfying the assumptions of Theorem~\ref{thm:MainThm_general}, then $\cZ(S^p \times D^{q+1}, \psi \times \theta)$ is invertible. 
\end{lemma}

There are two cases $q = 0, d-2$ and $0 < q <  d-2$. In the first case (say $q=0$) the bordism in question is an annulus $S^{d-2} \times I$, read as a bordism from $S^{d-2} \sqcup S^{d-2}$ to $\emptyset$. There is a dual annulus which goes the other way, and the composite gives $S^{d-2} \times S^1$, which as we have already seen takes an invertible value under $\cZ$. Since the objects $\emptyset$ and $S^{d-2} \sqcup S^{d-2}$ also take invertible values under $\cZ$, it follows from Lemma~\ref{lem:symmetric_mon_invert_mor} that both annuli take invertible values. (The dual annulus covers the case $q = d-2$). 

When $0 < q <  d-2$, then we instead consider the basic dimensional reduction of Section~\ref{sec:dimredux_basic}, in which we precompose $\cZ$ with the functor:
\begin{equation*}
	(S^p, \psi) \times (-): \Bord_{q+2}^\text{fr} \to \Bord_d^{(X, \xi)}.
\end{equation*}
This gives rise to a tangentially framed $(q+2)$-dimensional field theory. Under this functor the $(q+2)$-framed torus $(T^{q+1}, \theta_1 \times \theta_\text{Lie})$ gets mapped to $(S^p \times T^{q+1}, \psi \times \theta_1 \times \theta_\text{Lie})$, which we have already seen is invertible. Since $0 < q <  d-2$, we have that $3 < q+2 < d$, and so by induction we can apply Theorem~\ref{thm:MainThm_general} to conclude that this dimensionally reduced theory is invertible. It follows that any bordism in its image, such as $(S^p \times D^{q+1}, \psi \times \theta)$, takes an invertible value under $\cZ$. This establishes the the above lemma. 

\begin{corollary}\label{cor:surgery_invariance_tangential}
	Under the assumptions of Theorem~\ref{thm:MainThm_general}, suppose that $(N_1, \theta_1)$ and $(N_2, \theta_2)$ are parallel $(d-1)$-dimensional $(X, \xi$)-bordisms, such that $N_2$ is obtained from $N_1$ via $(X, \xi)$-surgery. Then $\cZ(N_1, \theta_1)$ is invertible if and only if $\cZ(N_2, \theta_2)$ is invertible. 
\end{corollary}

In fact Lemma~\ref{lem:knots_invertible} above shows a slightly stronger result. To say that $(N_1, \theta_1)$ and $(N_2, \theta_2)$ are related by `$(X, \xi)$-surgery' means that they are related by a finite sequence of moves, of the type to be explained. It suffices to assume that $N_2$ is obtained by one application of these moves. A move consists of the following: First an embedded $S^p \times D^{q+1}$ (with induced $(X, \xi)$-structure) is removed from $N_1$ to form a new $(X, \xi)$-bordism $\Sigma = N_1 \setminus S^p \times D^{q+1}$,
\begin{equation*}
	\Sigma: M_1  \to M_2 \sqcup S^p \times S^q
\end{equation*}
where $M_1$ and $M_2$ are the sources and targets of $N_1$. Next we compose $\Sigma$ with $D^{p+1} \times S^q$ to obtain $N_2$. By Lemma~\ref{lem:knots_invertible} the values of $\cZ(S^p \times D^{q+1})$ and $\cZ(D^{p+1} \times S^q)$ are invertible and hence $\cZ(N_1)$ is invertible if and only if $\cZ(\Sigma)$ is invertible if and only if $\cZ(N_2)$ is invertible. To actually count as \emph{surgery} we must further require that the $(X, \xi)$-structures on $S^p \times D^{q+1}$ and $D^{p+1} \times S^q$ must glue to an $(X, \xi)$-structure which extends to the handle $D^{p+1} \times D^{q+1}$. The above corollary does not actually need this requirement, and is valid for this `generalized $(X, \xi)$-surgery'. However if it is the case $(N_1, \theta_1)$ and $(N_2, \theta_1)$ are related by actual $(X, \xi)$-surgery, then they are connected by an $(X, \xi)$-bordism, and conversely an $(X,\xi)$-bordism can be used to obtain a sequence of surgeries relating $N_1$ and $N_2$ (for example by choosing a Morse function on this bordism). If in addition $M_1 = M_2 = \emptyset$ so that $N_1$ and $N_2$ are closed, then this means they represent the same element in the $(X, \xi)$-bordism group\footnote{$(X, \xi)$ is not a stable tangential structure, but the relation of $(X, \xi)$-bordism still makes sense for manifolds of dimension $k \leq d-1$, and yields abelian groups $\Omega_k^{(X, \xi)}$ defined in the usual way.}, $\Omega^{(X, \xi)}_{d-1}$. 

\begin{corollary}\label{cor:cor_all_1mor_invertible}
	If $N$ is any $(d-1)$-dimensional $(X, \xi)$-bordism from $M_1$ to $M_2$, then $\cZ(N)$ is invertible. 
\end{corollary}

The bordism $N$ is a 1-morphism in $\Bord_d^{(X, \xi)}$ and every 1-morphism in $\Bord_d^{(X, \xi)}$ has a (say, left) adjoint. Thus there exists another $(d-1)$-dimensional $(X, \xi)$-bordism $N^L$ from $M_2$ to $M_1$ and unit and counit morphisms witnessing the adjunction between $N$ and $N^L$. These are $(d-2)$-dimensional $(X, \xi)$-bordisms
\begin{align*}
	\eta: I \times M_2 \to N \circ N^L
	\varepsilon: N^L \circ N \to I \times M_1.
\end{align*}
Since the identity bordisms $I \times M_i$ are invertible in $\Bord_d^{(X, \xi)}$, they map to invertible values under $\cZ$. The existence of these $(d-2)$-dimensional bordisms means that $I \times M_2$ is related by surgery to $N \circ N^L$ and $I \times M_1$ is related by surgery to $N^L \circ N$. Hence by Cor.~\ref{cor:surgery_invariance_tangential}, $\cZ(N) \circ \cZ(N^L)$ and $\cZ(N^L) \circ \cZ(N)$ are invertible. Since invertible morphisms are closed under the 2-out-of-6 property, we have that both $\cZ(N^L)$ and $\cZ(N)$ are invertible.

Now we can return to and complete Step~2, showing that $\cZ(M, \theta)$ is invertible for all $(d-2)$-dimensional $(X, \xi)$-manifolds  $(M, \theta)$. It suffices to consider the case were $M$ is connected. Given such a manifold, we consider $M \times I$ as a bordism from $M \sqcup M$ to $\emptyset$. There is a unique $(X, \xi)$-structure on $M \times I$ which extends the $(X, \xi)$-structure on the first copy of $M$. On the second copy this determines a dual $(X, \xi)$-structure  $\overline{\theta}$ on $M$. We denote $\overline{M} = (M, \overline{\theta})$. The $(X, \xi)$-connect sum of $M$ and $\overline{M}$ yields the manifold $M \# \overline{M}$. This manifold is connected and satisfies $w_{d-2}(M \# \overline M) = 0$, and hence $\cZ(M \# \overline M)$ is invertible. 

Moreover there is a $(d-1)$-dimensional $(X, \xi)$-bordisms which witnesses the connect sum operation. This is a  higher dimensional analogs of the pair-of-pants bordism:
\begin{equation*}
	P: M \# \overline M \to M \sqcup \overline M.
\end{equation*}
As we have just seen, $\cZ(P): \cZ(M \# \overline M) \to \cZ(M) \otimes \cZ(\overline{M})$ is invertible, and since $\cZ(M \# \overline M)$ is invertible it follows that
\begin{equation*}
	\cZ(M) \otimes \cZ(\overline{M}) \cong \cZ(\overline{M}) \otimes \cZ({M})
\end{equation*}
is invertible, and hence $\cZ(M)$ is invertible.

Thus we have shown that $\cZ$ takes invertible values on all $(d-2)$-manifolds and all $(d-1)$-bordisms. Theorem~\ref{thm:MainThm_general} now follows directly from Lemma~\ref{lem:moving_up_gentangential}.
\end{proof}

%\begin{enumerate}
%	\item dimensional reduction with tangential structures
%	\begin{equation*}
%		(-) \times (S^1, \theta_\text{Lie}): \Bord_{d-1}^{(X, \xi)} \to \Bord_d^{(X, \xi)}
%	\end{equation*}
%	\begin{equation*}
%		(M, \theta) \times (-): \Bord_2^\text{fr} \to \Bord_d^{(X, \xi)}
%	\end{equation*}
%	\item Lemma~\ref{lem:surgery_invertibility}, need to use $(X, \xi)$-surgery.
%	\item $(X, \xi)$ is not a stable structure, but since it comes from a $d$-dimensional structure the monoid of $(d-1)$-dimensional $(X,\xi)$-manifolds up to bordism is still an abelian group. And reversal makes sense (is that true if $X$ has multiple components?). 
%	\item The tori may no longer be null-bordant. So $\overline{N}$ should not be interpreted as the inverse of $N$, but rather a manifold such that $N \sqcup \overline{N} \sim T^{d-1}$. 
%	\item ---- do we actually need all the components of $X$??
%	\item The last part of the argument is the same, where $\overline{W}$ is the reversal of $W$. 
%\end{enumerate}

\section{Extending Downward} \label{sec:extending_down}

Now we will show how to extend our previous results about $(\infty,2)$-categorical field theories to more extended $(\infty,n)$-categorical theories. We will need to set up some notation. Let $\Bord^{(X, \xi)}_{d;n}$ symmetric monoidal $(\infty, n)$-category whose objects are closed oriented $(d-n)$-manifolds, whose 1-morphsims are oriented $(d-n+1)$-dimensional bordisms, etc, up to dimension $d$, and where everything is equipped with an $(X, \xi)$-structure. Our previous results concerned the symmetric monoidal $(\infty, 2)$-category $\Bord^{(X, \xi)}_{d;2}$. 

\begin{theorem}\label{thm:main}
	Let $\cZ:\Bord^{(X, \xi)}_{d;n} \to \cC$ be an extended topological field theory valued in the symmetric monoidal $(\infty,n)$-category $\cC$. Assume that either $d\geq 3$ or that $(X, \xi)$ is spherophilic. Suppose that $n \geq 2$, and that for every $[x] \in \pi_0 X$,
	\begin{equation*}
		\cZ( T^{d-1}, x_* \theta_1 \times \theta_\text{Lie})
	\end{equation*}
	is invertible. Then $\cZ$ is invertible. 
\end{theorem}

\begin{proof}
	We will induct on the category number $n$. The base case $n = 2$ is covered by Theorem~\ref{thm:MainThm_general} and Proposition~\ref{pro:basecase_spherophilic}. So we assume that the above theorem holds for all $d$ and all $ k< n$ and we wish to show that it holds for  $k = n$. From our given topological field theory,
	\begin{equation*}
		\cZ: \Bord^{(X, \xi)}_{d;n} \to \cC
	\end{equation*}
we can extract two additional field theories.

First, out of any symmetric monoidal $(\infty,n)$-category, we can obtain a symmetric monoidal $(\infty,n-1)$-category by passing to the endomorphisms of the unit object. This is functorial and so the above functor induces a functor (also dentoed $\cZ$):
\begin{equation*}
	\cZ: \Hom_{\Bord^{(X, \xi)}_{d;n}}( \emptyset, \emptyset) \to \Hom_\cC(1,1).
\end{equation*}
The source $(\infty,n-1)$-category is precisely $\Bord^{(X, \xi)}_{d;n-1}$. Thus by induction, this restricted field theory is invertible. For example $\cZ(T^{d-2}, x_* \theta_1 \times \theta_\text{Lie})$ is invertible. 

Next, there is a functor
\begin{equation*}
	\Bord^{(X_{d-1}, \xi_{d-1})}_{d-1;n-1} \to \Bord^{(X, \xi)}_{d;n}.
\end{equation*}	
Here $(X_{d-1}, \xi_{d-1})$	is the pullback of $(X, \xi)$ to $BO(d-1)$ and defines a tangential structure for $(d-1)$-manifolds. The difference between $\Bord^{(X_{d-1}, \xi_{d-1})}_{d-1;n-1} $ and $\Bord^{(X, \xi)}_{d;n}$ is that the former includes manifolds only up to dimension $(d-1)$. The above functor simply includes the objects, morphisms, etc of $\Bord^{(X_{d-1}, \xi_{d-1})}_{d-1;n-1} $ into $\Bord^{(X, \xi)}_{d;n}$. By precomposing with this functor we get a field theory, which we again denote $\cZ$ 
\begin{equation*}
	\cZ: \Bord^{(X_{d-1}, \xi_{d-1})}_{d-1;n-1} \to \cC
\end{equation*}	
(which lands in the maximal sub-$(\infty, n-1)$-category of $\cC$). As we have seen $\cZ(T^{d-2}, x_* \theta_1 \times \theta_\text{Lie})$ is invertible, and if $d=3$, we have that $(X_2, \xi_2)$ is spherophilic by Example~\ref{Ex:restriction_spherophilia}. Hence by induction this restricted field theory is also invertible. In particular every object, 1-morphisms, etc, up to $(n-1)$-morphism of $\Bord^{(X, \xi)}_{d;n}$ takes an invertible value under $\cZ$.  

It follows that the only morphisms of $\Bord^{(X, \xi)}_{d;n}$ which could possibly take non-invertible values under $\cZ$ are the $(n-1)$-morphisms of the hom $(\infty, n-1)$-categories
\begin{equation*}
	\Hom_{\Bord^{(X, \xi)}_{d;n}}(M_1,M_2)
\end{equation*}
where $M_1, M_2$ are objects of $\Bord^{(X, \xi)}_{d;n}$, Moreover these $(n-1)$-morphisms also take invertible values if $M_1 = M_2 = \emptyset$.

Since $M_1$ is dualizable with dual $\overline{M}_1$, we have an equivalence of hom $(\infty, n-1)$-categories:
\begin{equation*}
	\Hom_{\Bord^{(X, \xi)}_{d;n}}(M_1,M_2) \simeq \Hom_{\Bord^{(X, \xi)}_{d;n}}(\emptyset,\overline{M}_1 \sqcup M_2).
\end{equation*}
This equivalence comes about by composing with the coevaluation of the duality between $M_1$ and $\overline{M}_1$. This equivalence is sometimes called the \emph{calculus of mates}. 

The functor $\cZ$, like any functor, preserves duality structures and hence the $(n-1)$-morphisms of $\Hom_{\Bord^{(X, \xi)}_{d;n}}(M_1,M_2)$ take invertible values under $\cZ$ precisely if the corresponding $(n-1)$-morphisms of $\Hom_{\Bord^{(X, \xi)}_{d;n}}(\emptyset,\overline{M}_1 \sqcup M_2)$ take invertible values under $\cZ$. Thus it is sufficient to show that the $(n-1)$-morphisms of $\Hom_{\Bord^{(X, \xi)}_{d;n}}(\emptyset,\overline{M}_1 \sqcup M_2)$ are invertible after applying $\cZ$. 

Suppose that we are given a morphism $f: \overline{M}_1 \sqcup M_2 \to \emptyset$ in $\Bord^{(X, \xi)}_{d;n}$. Composition with $f$ induces a functor:
\begin{equation*}
	f \circ (-): \Hom_{\Bord^{(X, \xi)}_{d;n}}(\emptyset,\overline{M}_1 \sqcup M_2) \to \Hom_{\Bord^{(X, \xi)}_{d;n}}(\emptyset, \emptyset).
\end{equation*}
We know that $\cZ(f)$ is invertible, and so composition with $\cZ(f)$ is an equivalence. Since the $(n-1)$-morphisms of $\Hom_{\Bord^{(X, \xi)}_{d;n}}(\emptyset, \emptyset)$ take invertible values under $\cZ$ it follows that the same is true of the $(n-1)$-morphisms of $\Hom_{\Bord^{(X, \xi)}_{d;n}}(\emptyset,\overline{M}_1 \sqcup M_2)$. Which is exactly what we set out to show. 

So our theorem will be proven if we can show that a morphism $f:\overline{M}_1 \sqcup M_2 \to \emptyset$ exists. If $\Hom_{\Bord^{(X, \xi)}_{d;n}}(\emptyset,\overline{M}_1 \sqcup M_2)$ is empty, then we are already done, vacuously. If $\Hom_{\Bord^{(X, \xi)}_{d;n}}(\emptyset,\overline{M}_1 \sqcup M_2)$ is non-empty, then there exists at least one morphism $g: \emptyset \overline{M}_1 \sqcup M_2$. We may obtain the desired $f$ as the, say, left-adjoint of $g$, $f = g^L$. 
\end{proof}

\bibliographystyle{plain}
\bibliography{invert_tori-bibliography}
\end{document}